\documentclass[11pt,reqno]{amsart}
\usepackage[a4paper]{geometry}

\usepackage{amsmath,amssymb,amsfonts,mathrsfs}
\usepackage{graphicx}
\usepackage{placeins}
\usepackage{xcolor}
\makeatletter
\def\hyper@nopatch@thm{}
\makeatother
\usepackage[colorlinks=true,linkcolor=blue,citecolor=blue,urlcolor=blue]{hyperref}

\usepackage{cite}

\newtheorem{theorem}{Theorem}[section]
\newtheorem{lemma}[theorem]{Lemma}
\newtheorem{proposition}[theorem]{Proposition}
\newtheorem{corollary}[theorem]{Corollary}

\theoremstyle{definition}
\newtheorem{definition}[theorem]{Definition}
\newtheorem*{definition*}{Definition}

\theoremstyle{remark}
\newtheorem{remark}{Remark}[section]

\numberwithin{equation}{section}

\begin{document}

\title[LORENTZIAN PRINCIPAL GEOMETRY FROM THE HYPERBOLIC DIRICHLET-TO-NEUMANN MAP]{Global Recovery of Lorentzian Principal Geometry from the Hyperbolic Dirichlet-to-Neumann Map}

\author{Lu Chen}
\address[Lu Chen]{Key Laboratory of Algebraic Lie Theory and Analysis of Ministry of Education, School of Mathematics and Statistics, Beijing Institute of Technology, Beijing, 100081, P.R. China; and MSU-BIT-SMBU Joint Research Center of Applied Mathematics, Shenzhen MSU-BIT University, Shenzhen 518172, P.R. China}
\email{chenlu5818804@163.com}

\author{Yan Jiang}
\address[Yan Jiang]{Department of Mathematics, City University of Hong Kong, Hong Kong SAR, China}
\email{yjian24@cityu.edu.hk}

\author{Hongyu Liu}
\address[Hongyu Liu]{Department of Mathematics, City University of Hong Kong, Hong Kong SAR, China}
\email{hongyu.liuip@gmail.com, hongyliu@cityu.edu.hk}

\author{Longyue Tao}
\address[Longyue Tao]{Department of Mathematics, City University of Hong Kong, Hong Kong SAR, China}
\email{sdyctly@163.com, longyue.tao@my.cityu.edu.hk}



\date{August 29, 2026}

\hypersetup{
  pdfauthor={Lu Chen, Yan Jiang, Hongyu Liu, and Longyue Tao},
  pdftitle={Global Recovery of Lorentzian Principal Geometry from the Hyperbolic Dirichlet-to-Neumann Map},
  pdfsubject={Global recovery of Lorentzian principal geometry from the hyperbolic Dirichlet-to-Neumann map},
  pdfkeywords={inverse hyperbolic problem, Lorentzian geometry, Riemannian metric, time-dependent wave speed, finite-time Dirichlet-to-Neumann map, global uniqueness}
}

\begin{abstract}
We prove that the full Dirichlet-to-Neumann map on a finite time interval globally determines the Lorentzian metric \(\mathbf G_{g,c}=-c(x,t)^2\mathrm{d}t^2+g(x)\) encoded by the principal symbol of the wave operator.
The result holds in spatial dimensions \(n\geq3\) for multiplicatively separable wave speeds \(c(x,t)=a(x)b(t)\), provided the accumulated effective time during the experiment exceeds the maximal travel time from the boundary to the interior and back.
The data determine both the unknown Riemannian metric \(g(x)\) and the wave speed \(c(x,t)\) throughout the observation cylinder, up to a spatial change of coordinates that fixes the boundary and leaves the measured time unchanged.
In particular, the temporal factor \(b\) is not prescribed but is recovered from the same measurements.
We also construct several examples demonstrating that the dimensional and visibility conditions for uniqueness are sharp.

\end{abstract}

\maketitle
\enlargethispage{2pt}
\section{Introduction and main result}
\label{sec:introduction}

We call the Lorentzian metric determined by the homogeneous second-order symbol of a scalar wave operator its Lorentzian principal geometry.
We study the global recovery of this geometry from the finite-time Dirichlet-to-Neumann map on the full lateral boundary.
Let \((M,g)\) be a smooth compact connected Riemannian manifold with smooth nonempty boundary.
Let \(T>0\), and let \(c\in C^\infty(M\times[0,T])\) be positive.
For a smooth boundary source \(f\) vanishing near \(t=0\), let \(u\) solve
\begin{equation}
\begin{cases}
c(x,t)^{-2}\partial_t^2u-\Delta_gu=0
& \text{in }M\times(0,T),\\
u|_{t=0}=\partial_tu|_{t=0}=0
& \text{in }M,\\
u=f
& \text{on }\partial M\times(0,T).
\end{cases}
\label{eq:main-forward-problem}
\end{equation}
On the spacetime cylinder \(M\times(0,T)\), the principal symbol of the wave operator in \eqref{eq:main-forward-problem} is
\begin{equation}
\sigma_2\bigl(c^{-2}\partial_t^2-\Delta_g\bigr)(x,t;\xi,\tau)
=
-c(x,t)^{-2}\tau^2+|\xi|_g^2.
\label{eq:lorentzian-principal-symbol}
\end{equation}
This is the dual quadratic form of the Lorentzian metric
\begin{equation}
\mathbf G_{g,c}
=
-c(x,t)^2\mathrm{d}t^2+g(x)
\quad\text{on }M\times(0,T).
\label{eq:principal-lorentzian-metric}
\end{equation}
Thus, in the present problem, \(\mathbf G_{g,c}\) is precisely the Lorentzian principal geometry of the wave operator.
The adjective \emph{principal} emphasizes that this geometry is fixed by the homogeneous second-order part of the operator.
The operator in \eqref{eq:main-forward-problem} equals \(-\Box_{\mathbf G_{g,c}}\) plus first-order terms determined by \(g\) and \(c\).
Consequently, the characteristics and finite propagation of \eqref{eq:main-forward-problem} are governed by the null geometry of \(\mathbf G_{g,c}\).
The finite-time Dirichlet-to-Neumann map is
\begin{equation}
\Lambda_{g,c}^Tf
=
\partial_{\nu_g}u|_{\partial M\times(0,T)}.
\label{eq:main-dn-map}
\end{equation}
Here \(\nu_g\) is the outward unit normal of \((M,g)\).
We ask whether \(\Lambda_{g,c}^T\) determines \(\mathbf G_{g,c}\), equivalently \((g,c)\) within the product structure \eqref{eq:principal-lorentzian-metric}.
As in geometric inverse boundary problems, the data are invariant under boundary-fixing diffeomorphisms: if \(\Phi:M\to M\) satisfies \(\Phi|_{\partial M}=\operatorname{Id}\) and \(\widehat\Phi=\Phi\times\operatorname{Id}_{(0,T)}\), then \(\Lambda_{\Phi^*g,\,c\circ\widehat\Phi}^T=\Lambda_{g,c}^T\).
Thus recovery within \eqref{eq:principal-lorentzian-metric} can only be expected up to a product diffeomorphism \(\Phi\times\operatorname{Id}_{(0,T)}\).

We consider a structured class that retains both an unknown spatial geometry and an unknown temporal propagation law.
Let \(\mathcal C^T(M)\) consist of wave speeds of the form
\begin{equation}
c(x,t)=a(x)b(t).
\label{eq:main-separable-speed}
\end{equation}
Here \(a\in C^\infty(M)\) and \(b\in C^\infty([0,T])\) are strictly positive and unknown, and \(b(0)=1\) fixes only their constant scaling ambiguity.

Define the effective time and the metric governing travel times by
\[
 s=B(t):=\int_0^t b(r)\,\mathrm{d}r,
\qquad
S_T=B(T),
\qquad
h=a^{-2}g.
\]
Since \(b>0\), \(B\) is a smooth increasing diffeomorphism from \([0,T]\) onto \([0,S_T]\).
In the coordinates \((x,s)\), the Lorentzian metric in \eqref{eq:principal-lorentzian-metric} becomes
\begin{equation}
\mathbf G_{g,c}
=
a(x)^2\bigl(-\mathrm{d}s^2+h(x)\bigr).
\label{eq:conformally-static-lorentzian-form}
\end{equation}
Thus the class is conformal to a product Lorentzian geometry in the effective time \(s\).
The quantity \(S_T\) measures the propagation available during the experiment, and \(h\) governs travel times.
The physical duration \(T\) alone does not determine the available propagation.
The propagation clock \(B\), the spatial geometry \(g\), and the spatial factor \(a\) are all unknown and must be recovered from the same boundary map.

\subsection{Main theorem}
\label{subsec:problem-data}

For a Riemannian metric \(h\) on \(M\), write
\begin{equation}
\operatorname{rad}(h)
=
\max_{x\in M}\operatorname{dist}_h(x,\partial M)
\label{eq:main-filling-radius}
\end{equation}
for its boundary filling radius.
For \(h=a^{-2}g\), \(\operatorname{rad}(h)\) is the largest travel time from the boundary to a point, and \(2\operatorname{rad}(h)\) is the time required to reach the deepest point and return to the boundary.

\begin{theorem}
\label{thm:simultaneous-recovery}
Let \(T>0\) and \(n\geq3\).
For \(j=1,2\), let \((M_j,g_j)\) be a smooth compact connected \(n\)-dimensional Riemannian manifold with smooth nonempty boundary and let
\[
c_j(x,t)=a_j(x)b_j(t),
\]
where \(a_j\in C^\infty(M_j)\) and \(b_j\in C^\infty([0,T])\) are strictly positive and \(b_j(0)=1\).
Set
\[
\mathbf G_j=-c_j(x,t)^2\mathrm{d}t^2+g_j(x)
\quad\text{on }M_j\times(0,T).
\]
Let \(\varphi:\partial M_1\to\partial M_2\) be a diffeomorphism, and let the scalar pushforward, with no Jacobian factor, be
\begin{equation}
(P_\varphi f)(z,t)=f(\varphi^{-1}(z),t).
\label{eq:boundary-pushforward}
\end{equation}
Assume that the full Dirichlet-to-Neumann maps satisfy
\begin{equation}
\Lambda_{g_1,c_1}^T
=
P_\varphi^{-1}\Lambda_{g_2,c_2}^TP_\varphi
\label{eq:boundary-identification}
\end{equation}
and that each model satisfies the visibility condition
\begin{equation}
\int_0^T b_j(t)\,\mathrm{d}t
>
2\operatorname{rad}\!\bigl(a_j^{-2}g_j\bigr),
\qquad
j=1,2.
\label{eq:main-finite-visibility}
\end{equation}
Then \(\varphi\) extends to a diffeomorphism \(\Phi:M_1\to M_2\) such that, for \(\widehat\Phi=\Phi\times\operatorname{Id}_{(0,T)}\),
\begin{equation}
\widehat\Phi^*\mathbf G_2=\mathbf G_1.
\label{eq:lorentzian-recovery-conclusion}
\end{equation}
Equivalently, the same diffeomorphism satisfies
\begin{equation}
\Phi^*g_2=g_1,
\qquad
c_2(\Phi(x),t)=c_1(x,t)
\quad\text{on }M_1\times[0,T].
\label{eq:simultaneous-recovery-conclusion}
\end{equation}
\end{theorem}

Theorem~\ref{thm:simultaneous-recovery} is our main result.
We next record the positive consequences of the recovery argument on a prescribed Euclidean background, followed by the obstructions that delimit the main result and the associated observation scale.

\begin{corollary}
\label{cor:euclidean-recovery-consequences}
Let \(T>0\), let \(n\geq2\), and let \(\Omega\subset\mathbb R^n\) be a smooth bounded connected domain.
For \(j=1,2\), let
\[
c_j(x,t)=a_j(x)b_j(t),
\]
where \(a_j\in C^\infty(\overline\Omega)\) and \(b_j\in C^\infty([0,T])\) are strictly positive and \(b_j(0)=1\).
Set
\[
h_j=a_j^{-2}\mathrm{d}x^2,
\qquad
r_\Omega
=
\max_{x\in\overline\Omega}
\operatorname{dist}_{\mathrm{d}x^2}(x,\partial\Omega),
\]
and assume that
\[
\Lambda_{\mathrm{d}x^2,c_1}^T
=
\Lambda_{\mathrm{d}x^2,c_2}^T.
\]
Each of the following conditions is sufficient for recovery.
\begin{enumerate}
\item[(i)] If
\[
\int_0^T b_j(t)\,\mathrm{d}t
\geq
2\operatorname{rad}(h_j),
\qquad j=1,2,
\]
then \(c_1=c_2\) on \(\overline\Omega\times[0,T]\).

\item[(ii)] If \(a_j\geq a_*>0\) on \(\overline\Omega\) and
\[
\int_0^T b_j(t)\,\mathrm{d}t
\geq
\frac{2r_\Omega}{a_*},
\qquad j=1,2,
\]
then \(c_1=c_2\) on \(\overline\Omega\times[0,T]\).
\end{enumerate}
Thus recovery on the prescribed Euclidean background holds in every dimension \(n\geq2\), and part~\textup{(i)} includes equality in its visibility condition; part~\textup{(ii)} gives a sufficient criterion uniform over all coefficients with the stated lower bound.
\end{corollary}

The preceding corollary records the positive consequences of fixing the background geometry.
The next theorem collects the obstructions that delineate the scope of Theorem~\ref{thm:simultaneous-recovery} and show that its dimension and visibility conditions are sharp.
The theorem summarizes the individual results proved separately in Section~\ref{sec:consequences-sharpness}.

\begin{theorem}
\label{thm:sharpness-obstructions}
The following statements hold within the multiplicatively separable class \eqref{eq:main-separable-speed}.
\begin{enumerate}
\item[(i)] \emph{Two-dimensional conformal obstruction.}
Let \((M,g_1)\) be a smooth compact connected surface with smooth nonempty boundary, let \(a_1\in C^\infty(M)\) and \(b\in C^\infty([0,\infty))\) be strictly positive with \(b(0)=1\), and set \(c_1=a_1b\).
If \(\kappa\in C^\infty(M)\) is positive with \(\kappa|_{\partial M}=1\), then
\[
g_2=\kappa^2g_1,
\qquad
c_2=\kappa c_1
\]
satisfy \(\Lambda_{g_1,c_1}^T=\Lambda_{g_2,c_2}^T\) for every finite \(T>0\).
If \(\kappa\geq1\) and \(\kappa>1\) on a nonempty interior open set, the two pairs are not related by any boundary-fixing diffeomorphism.

\item[(ii)] \emph{Sharpness of the uniform Euclidean criterion.}
Let \(n\geq2\) and let \(\Omega\subset\mathbb R^n\) be a smooth bounded connected domain, with \(r_\Omega\) as in Corollary~\ref{cor:euclidean-recovery-consequences}.
For every \(0<C<2\) and \(T>0\), there exist \(a_*>0\), distinct speeds \(c_j=a_jb\) with a common strictly positive nonconstant temporal factor \(b\in C^\infty([0,T])\), \(b(0)=1\), and \(a_j\geq a_*\), such that
\[
\int_0^T b(t)\,\mathrm{d}t
>
\frac{Cr_\Omega}{a_*},
\qquad
\Lambda_{\mathrm{d}x^2,c_1}^T
=
\Lambda_{\mathrm{d}x^2,c_2}^T.
\]
Consequently, no universal constant \(C<2\) can replace \(2\) in Corollary~\ref{cor:euclidean-recovery-consequences}\textup{(ii)}.

\item[(iii)] \emph{Finite total effective time over infinite physical time.}
For every \(n\geq2\) and every smooth bounded connected domain \(\Omega\subset\mathbb R^n\), there exist distinct speeds \(c_j=a_jb\) on \(\overline\Omega\times[0,\infty)\), with a common strictly positive normalized temporal factor satisfying
\[
\int_0^\infty b(t)\,\mathrm{d}t<\infty,
\]
such that
\[
\Lambda_{\mathrm{d}x^2,c_1}^T
=
\Lambda_{\mathrm{d}x^2,c_2}^T
\qquad
\text{for every finite }T>0.
\]
They may be chosen with \(a_1=a_2\) near \(\partial\Omega\), \(0<a_2\leq a_1\), and \(a_1\not\equiv a_2\).

\item[(iv)] \emph{Endpoint obstruction on unknown manifolds.}
For every \(n\geq3\) and \(R>0\), there exist smooth compact connected Riemannian manifolds \((M_j,g_j)\) and a boundary diffeomorphism \(\varphi:\partial M_1\to\partial M_2\) such that
\[
\operatorname{rad}(g_1)=\operatorname{rad}(g_2)=R,
\qquad
\Lambda_{g_1,1}^{2R}
=
P_\varphi^{-1}\Lambda_{g_2,1}^{2R}P_\varphi
\]
on the open observation interval \((0,2R)\), but no isometry \(F:(M_1,g_1)\to(M_2,g_2)\) satisfies \(F|_{\partial M_1}=\varphi\).
\end{enumerate}
\end{theorem}

\begin{remark}
\label{rem:sharpness-main-theorem}
The detailed results in Section~\ref{sec:consequences-sharpness} explain the scope and sharpness of Theorem~\ref{thm:simultaneous-recovery}.
Theorem~\ref{thm:two-dimensional-conformal-obstruction} and Corollary~\ref{cor:fixed-euclidean-recovery} locate the restriction \(n\geq3\): simultaneous recovery of an unknown geometry and wave speed fails in dimension two, whereas recovery on a prescribed Euclidean background remains valid for \(n\geq2\).
Corollary~\ref{cor:fixed-euclidean-recovery} and Theorem~\ref{thm:endpoint-gluing-obstruction} show that equality in \eqref{eq:main-finite-visibility} is sufficient on a prescribed Euclidean background but not on an unknown manifold.
Finally, Corollary~\ref{cor:uniform-finite-time} and Theorem~\ref{thm:causal-obstruction} identify the sharp constant \(2\) in the uniform Euclidean criterion, while Corollary~\ref{cor:infinite-time-threshold} shows that physical duration alone cannot replace total effective time.
\end{remark}

\subsection{Motivation and related results}\label{subsec:motivation-related-results}

The global inverse problem of recovering a smooth time-dependent Lorentzian metric from linear boundary data remains largely open \cite{kurylev2018inverse,lassas2018inverse}.
In unrestricted coefficient classes, however, global uniqueness is impossible: counterexamples are known for both partial and full boundary measurements \cite{liimatainen2022counterexamples,oksanen2026counterexamples}.
Thus appropriate structural assumptions on the Lorentzian geometry are indispensable.
Accordingly, we consider the class of Lorentzian metrics \(\mathbf G_{g,c}\) associated with the multiplicatively separable speeds in \eqref{eq:main-separable-speed}, allowing both the background Riemannian metric \(g\) and the full wave speed \(c=ab\) to be unknown.
In spatial dimensions \(n\geq3\), Theorem~\ref{thm:simultaneous-recovery} determines \(\mathbf G_{g,c}\), equivalently both \(g\) and \(c\), modulo the natural time-preserving spatial gauge, from a single finite-time Dirichlet-to-Neumann map under the visibility condition \eqref{eq:main-finite-visibility}.
Corollary~\ref{cor:euclidean-recovery-consequences} gives the corresponding extensions when the Euclidean background is prescribed, while Theorem~\ref{thm:sharpness-obstructions} proves that both the dimensional restriction and the visibility requirements in Theorem~\ref{thm:simultaneous-recovery} are sharp.

For wave operators with smooth Lorentzian principal metrics, boundary measurements have been shown, under suitable assumptions, to determine boundary jets, null scattering information, and ray transforms of lower-order terms \cite{stefanov2018inverse,oksanen2024inverse}.
Passing from such information to interior recovery is a rigidity problem.
Lightlike strict convexity yields boundary determination of the metric jet, while analyticity gives local rigidity, in both cases up to the natural gauge \cite{stefanov2025boundary}.
Global smooth rigidity relative to Minkowski space is known for compactly supported perturbations from measurements over the entire time axis \cite{oksanen2026rigidity}, and semiglobal uniqueness is available when one perturbation of Minkowski space is small while the other may be large \cite{oksanen2026semiglobal}.
Other global results for unknown static metrics with time-dependent lower-order terms, or for time-dependent Lorentzian principal parts, rely on time analyticity together with long-time geometric control assumptions \cite{eskin2007inverse,eskin2017general}.

Complementary approaches change either the unknowns or the data.
When the propagation geometry is prescribed, time-dependent scalar, magnetic, and other lower-order coefficients can be recovered under geometric or gauge assumptions \cite{stefanov1989uniqueness,salazar2013determination,kian2019recovery,feizmohammadi2021recovery,feizmohammadi2022global,alexakis2022lorentzian,alexakis2025lorentzian}.
Another approach formulates the recovery of time-dependent coefficients as a dynamic operator-identification problem and studies regularity and differentiability of the parameter-to-solution map, local ill-posedness, and regularized reconstruction \cite{gerken2020dynamicregularity,gerken2020dynamic}.
Nonlinear source-to-solution data provide a different route to Lorentzian geometry through wave interactions \cite{kurylev2018inverse,lassas2018inverse,hintz2022dirichlet}.
By contrast, the present work recovers the Lorentzian metric \(\mathbf G_{g,c}\) from the full Dirichlet-to-Neumann map on a single finite observation interval; equivalently, it determines both \(g\) and \(c=ab\), including the temporal factor \(b\), without time analyticity, smallness, simplicity, or nontrapping assumptions.

A related precedent for multiplicative separation is the one-dimensional first-order model of \cite{wall1999inverse}, where a downstream trace determines one factor in a separable speed when the other is prescribed.
The present problem instead concerns recovery of the Lorentzian metric associated with a multidimensional second-order wave equation, with the background geometry and both separated factors simultaneously unknown.

Classical finite-time boundary control for time-independent waves shows that boundary responses on the round-trip scale determine the visible Riemannian or acoustic structure \cite{belishev1997boundary,dehoop2018recovery}.
Related reconstruction, continuation, and stability results are developed in \cite{belishev1991boundary,belishev1992reconstruction,katchalov2001inverse,katchalov2004equivalence,katchalov2004energy,kurylev2002hyperbolic,stefanov1998stability,stefanov2005stable}.
For the separable time dependence considered here, however, a change of time variable alone does not convert the measured response into the stationary data used in these results.
The same measurements must also determine the temporal propagation law, which, combined with the recovered spatial information, yields the original Lorentzian metric in dimensions \(n\geq3\).

Taken together, our results establish global finite-time recovery for a genuinely time-dependent class of Lorentzian metrics in which neither the background geometry nor the full separable wave speed is prescribed.
They also delineate how the prescribed background, the spatial dimension, and the available effective propagation time govern the scope and limits of recovery.

\subsection{Organization of the paper}
\label{subsec:organization}

The argument is developed in four parts.
Section~\ref{sec:boundary-data-correspondence} recovers the temporal propagation law and establishes an information-preserving causal correspondence between the original time-dependent boundary map and the weighted response used for spatial reconstruction.
Section~\ref{sec:finite-time-recovery} establishes finite-time recovery of the weighted geometry \((M,h,\varrho)\) governing travel times under the visibility condition.
Section~\ref{sec:completion-main-theorem} recovers the original metric and the entire time-dependent wave speed, completing the proof of Theorem~\ref{thm:simultaneous-recovery}.
Section~\ref{sec:consequences-sharpness} establishes the scope and sharpness of
the result through the two-dimensional conformal obstruction, recovery on a
prescribed Euclidean background, the optimal uniform observation scale, and
the critical endpoint obstruction on unknown manifolds.

\section{Recovery of the propagation clock and exact correspondence of boundary data}
\label{sec:boundary-data-correspondence}

The boundary principal symbol determines \(g_{\partial M}\) and \(c|_{\partial M\times(0,T)}\), and the normalization \(b(0)=1\) therefore recovers the temporal factor \(b\).
It consequently determines the propagation clock \(B\) and the accumulated effective time
\[
B(t)=\int_0^t b(r)\,\mathrm{d}r,
\qquad
S_T=B(T).
\]
The change of clock alone does not produce a stationary response: it leaves a temporal potential determined by \(b\).
We therefore construct invertible causal operators, also determined by the recovered \(b\), that preserve the boundary source space and transport the full boundary map to and from the stationary response.
This yields the exact two-sided correspondence
\[
\Lambda_{g,c}^T
\quad\longleftrightarrow\quad
\bigl(g_{\partial M},b,R_{g,a}^{S_T}\bigr).
\]
Here \(R_{g,a}^{S_T}\) is the Dirichlet-to-Neumann map for \(z_{ss}-a^2\Delta_gz=0\) on \((0,S_T)\).
In invariant form, \((g_{\partial M},R_{g,a}^{S_T})\) determines the response \(\mathcal B_{h,\varrho}^{S_T}\) of the weighted geometry
\[
h=a^{-2}g,
\qquad
\varrho=a^{n-2}.
\]
This response provides the input for Section~\ref{sec:finite-time-recovery}.
The construction culminates in Proposition~\ref{prop:boundary-data-correspondence}, which proves that the passage is causal, invertible, and information preserving.

\subsection{Functional setting and causality}

We first establish smooth well-posedness and causality of the forward problem on a boundary source space stable under restriction of the terminal time.
These properties provide the functional framework for the boundary determination in Subsection~\ref{subsec:boundary-principal-symbol}.

For \(T>0\), set
\[
Q_T=M\times(0,T),
\qquad
\Sigma_T=\partial M\times(0,T).
\]
For any terminal time \(\tau>0\), define the admissible smooth control space
\[
\mathcal D_\tau
=
\left\{
f\in C^\infty(\partial M\times[0,\tau])
:
f(\cdot,t)=0\text{ for }0\leq t<\varepsilon_f
\text{ for some }\varepsilon_f>0
\right\}.
\]
Because \(f\) vanishes for all sufficiently small \(t\), all its time derivatives vanish at \(t=0\); hence the boundary values agree to every order with the zero initial data.
No condition is imposed near \(t=\tau\).
For a positive \(c\in C^\infty(M\times[0,T])\) and \(f\in\mathcal D_T\), let \(u^f\) denote the solution of \eqref{eq:main-forward-problem}.
We regard the map \(\Lambda_{g,c}^T\) in \eqref{eq:main-dn-map} as an operator on \(\mathcal D_T\).

\begin{proposition}
\label{prop:forward-well-posedness-causality}
For every \(f\in\mathcal D_T\), problem \eqref{eq:main-forward-problem} has a unique solution \(u^f\in C^\infty(M\times[0,T])\), and the map
\[
\Lambda_{g,c}^T
:
\mathcal D_T
\longrightarrow
\mathcal D_T
\]
is well defined.
If \(0<T_0<T\), then
\begin{equation}
\left.
\Lambda_{g,c}^Tf
\right|_{\partial M\times(0,T_0)}
=
\Lambda_{g,c}^{T_0}
\left(
f|_{\partial M\times[0,T_0]}
\right),
\label{eq:finite-time-restriction}
\end{equation}
where the wave speed on the right is the restriction of \(c\) to \(M\times[0,T_0]\).
Wave speeds and inputs that agree up to time \(T_0\) produce the same solution and Neumann trace up to time \(T_0\).
\end{proposition}

\begin{proof}
{
To justify the Dirichlet-to-Neumann map and its causal restriction, choose boundary normal coordinates in a collar and multiply the normal extension of \(f(y,t)\) by a cutoff in the normal variable.
This gives a smooth interior lifting \(F\), supported in the collar, whose boundary trace is \(f\) and which vanishes for all sufficiently small \(t\).
More explicitly, let \(\Psi:\partial M\times[0,\delta)\to U_\delta\subset M\) be a boundary normal collar parametrization with \(\Psi(y,0)=y\), and choose \(\chi\in C_c^\infty([0,\delta))\) such that \(\chi=1\) near \(r=0\).
Define
\[
F(\Psi(y,r),t)=\chi(r)f(y,t),
\qquad
F(x,t)=0\quad\text{for }x\in M\setminus U_\delta.
\]
Then \(F\in C^\infty(M\times[0,T])\), \(F|_{\Sigma_T}=f\), and \(F\) is supported in the collar and vanishes on the same initial time interval as \(f\).
}
Set \(w=u^f-F\).
Then \(w\) has homogeneous Dirichlet data and satisfies
\[
c^{-2}w_{tt}-\Delta_gw=c^{-2}G_F,
\qquad
G_F=-F_{tt}+c^2\Delta_gF.
\]
Multiplication by \(w_t\), integration over \(M\), and the homogeneous boundary condition give
\begin{equation}
\frac{\mathrm{d}}{\mathrm{d}t}E_w(t)
=
\int_M c^{-2}G_F w_t\,\mathrm{d}V_g
-
\int_M c^{-3}c_t|w_t|^2\,\mathrm{d}V_g,
\qquad
E_w(t)
=
\frac12\int_M
\left(c^{-2}|w_t|^2+|\nabla_g w|_g^2\right)\mathrm{d}V_g.
\label{eq:forward-energy-identity}
\end{equation}
{
Set
\[
c_-:=\min_{M\times[0,T]}c>0,
\qquad
c_+:=\max_{M\times[0,T]}c.
\]
The kinetic part of the energy satisfies
\[
\|w_t(\cdot,t)\|_{L^2(M)}^2
\leq
2c_+^2E_w(t).
\]
Hence, by the Cauchy--Schwarz and Young inequalities,
\begin{align*}
\left|
\int_M c^{-2}G_Fw_t\,\mathrm{d}V_g
\right|
&\leq
c_-^{-2}\|G_F(\cdot,t)\|_{L^2(M)}
\|w_t(\cdot,t)\|_{L^2(M)}\\
&\leq
C E_w(t)+C\|G_F(\cdot,t)\|_{L^2(M)}^2,
\end{align*}
while
\[
\left|
\int_M c^{-3}c_t|w_t|^2\,\mathrm{d}V_g
\right|
\leq
c_-^{-3}\|c_t\|_{L^\infty(M\times[0,T])}
\|w_t(\cdot,t)\|_{L^2(M)}^2
\leq
C E_w(t).
\]
Combining these estimates with \eqref{eq:forward-energy-identity} gives
\[
\frac{\mathrm{d}}{\mathrm{d}t}E_w(t)
\leq
C E_w(t)+C\|G_F(\cdot,t)\|_{L^2(M)}^2,
\]
where \(C\) depends only on \(c_-\), \(c_+\), and
\(\|c_t\|_{L^\infty(M\times[0,T])}\).
Since \(E_w(0)=0\), the integral form of Gronwall's inequality yields
\[
E_w(t)
\leq
C\int_0^t e^{C(t-s)}
\|G_F(\cdot,s)\|_{L^2(M)}^2\,\mathrm{d}s,
\qquad 0\leq t\leq T,
\]
and therefore
\begin{equation}
\sup_{0\leq t\leq T}E_w(t)
\leq
C_T\int_0^T\|G_F(\cdot,t)\|_{L^2(M)}^2\,\mathrm{d}t.
\label{eq:forward-energy-estimate}
\end{equation}
}
The same identity applied to the difference of two solutions proves uniqueness.
For existence, let \(\{\phi_k\}_{k\geq1}\) be an \(L^2(M,\mathrm{d}V_g)\)-orthonormal basis of Dirichlet eigenfunctions of \(-\Delta_g\), and seek
\[
w_N(x,t)=\sum_{k=1}^N d_{N,k}(t)\phi_k(x).
\]
Testing against \(\phi_i\), \(1\leq i\leq N\), gives
\[
\int_M c^{-2}w_{N,tt}\phi_i\,\mathrm{d}V_g
+
\int_M\langle\nabla_gw_N,\nabla_g\phi_i\rangle_g\,\mathrm{d}V_g
=
\int_M c^{-2}G_F\phi_i\,\mathrm{d}V_g.
\]
The mass matrix
\[
\mathsf M_N(t)_{ij}
=
\int_M c(x,t)^{-2}\phi_i(x)\phi_j(x)\,\mathrm{d}V_g
\]
is smooth and uniformly positive definite on \([0,T]\), so this system has a unique smooth solution with zero initial data.
Estimate \eqref{eq:forward-energy-estimate} is uniform in \(N\).
Weak compactness and passage to the limit in the Galerkin identities integrated in time give
\[
w\in L^\infty(0,T;H_0^1(M)),
\qquad
w_t\in L^\infty(0,T;L^2(M)).
\]
The limiting equation then gives \(w_{tt}\in L^2(0,T;H^{-1}(M))\), so \(w\) is an energy solution.

It remains to obtain the asserted smoothness.
On the initial interval on which \(F=0\), one has \(G_F=0\), so \eqref{eq:forward-energy-estimate} gives \(w=0\) there.
Consequently, every time derivative of \(w\) has zero initial displacement and velocity.
For \(m\geq0\), differentiating the equation \(m\) times and writing \(v_m=\partial_t^m w\) gives
\begin{equation}
\begin{aligned}
c^{-2}\partial_t^2v_m-\Delta_gv_m
&=H_m,\\
H_m
&=\partial_t^m(c^{-2}G_F)\\
&\quad-
\sum_{k=1}^m
\binom{m}{k}
\partial_t^k(c^{-2})
\partial_t^{m-k+2}w.
\end{aligned}
\label{eq:forward-differentiated-equation}
\end{equation}
These differentiated identities are justified by applying the energy estimate first to time difference quotients, or equivalently to the differentiated Galerkin systems before passing to the limit.
In the term with \(k=1\), the last factor is \(\partial_tv_m\), and the terms with \(k\geq2\) contain only \(v_m\) or lower time derivatives.
Poincar\'e's inequality controls \(\|v_m\|_{L^2(M)}\) by \(\|\nabla_gv_m\|_{L^2(M)}\).
Multiplication of \eqref{eq:forward-differentiated-equation} by \(\partial_tv_m\), followed by Young's and Gronwall's inequalities, therefore gives inductively
\[
\partial_t^m w\in C([0,T];H_0^1(M)),
\qquad
\partial_t^{m+1}w\in C([0,T];L^2(M))
\]
for every \(m\geq0\).
The elliptic identity
\[
-\Delta_g\partial_t^m w
=
\partial_t^m(c^{-2}G_F)
-
\partial_t^m(c^{-2}w_{tt})
\]
and homogeneous Dirichlet boundary data now give, by elliptic regularity and induction in the spatial Sobolev order,
\[
\partial_t^m w\in C([0,T];H^r(M))
\]
for every pair of integers \(m,r\geq0\).
Sobolev embedding in the spatial variables then yields \(w\in C^\infty(M\times[0,T])\), and hence \(u^f=w+F\) has the same regularity.

\enlargethispage{\baselineskip}
If \(c_1=c_2\) and \(f_1=f_2\) on \([0,T_0]\), then \(v=u_1-u_2\) has zero initial and Dirichlet data and solves the homogeneous equation with the common wave speed there.
The energy estimate gives \(v=0\), and smoothness gives \(\partial_{\nu_g}u_1=\partial_{\nu_g}u_2\) on the same interval.
This proves \eqref{eq:finite-time-restriction}.
If \(f=0\) on an initial interval, uniqueness gives \(u^f=\partial_{\nu_g}u^f=0\) there, so \(\Lambda_{g,c}^T\) maps \(\mathcal D_T\) into itself.

The proof is complete.
\end{proof}

The first inverse step is boundary determination.

\subsection{Boundary symbol and recovery of the propagation clock}
\label{subsec:boundary-principal-symbol}

The elliptic principal symbol of the localized Dirichlet-to-Neumann map determines the boundary metric \(g_{\partial M}\) and the boundary value \(c|_{\Sigma_T}\).
In the elliptic region, the dependence of the scalar normal root on the spatial and temporal covectors separates \(g_{\partial M}\) from \(c|_{\Sigma_T}\).
\begin{lemma}
\label{lem:boundary-symbol}
Assume that \(\dim M=n\geq2\), and let \(c\in C^\infty(M\times[0,T])\) be positive.
Fix \((y_0,t_0)\in\Sigma_T\) and a nonzero covector
\[
\eta+\omega\,\mathrm{d}t
\in
T_{(y_0,t_0)}^*\Sigma_T\setminus\{0\},
\]
where \(\eta\in T_{y_0}^*\partial M\) is the spatial tangential covector and \(\omega\in\mathbb R\) is the temporal frequency.
Let \(g_{\partial M}\) denote the metric induced by \(g\) on \(\partial M\).
Suppose that this covector belongs to the elliptic region
\begin{equation}
|\eta|_{g_{\partial M}(y_0)}^2
>
\frac{\omega^2}{c(y_0,t_0)^2},
\label{eq:elliptic-region}
\end{equation}
where \(|\eta|_{g_{\partial M}}\) is the norm on covectors induced by the inverse boundary metric.
After microlocalization in a sufficiently small conic neighborhood of \((y_0,t_0;\eta,\omega)\), the Dirichlet-to-Neumann map \(\Lambda_{g,c}^T\) is a classical pseudodifferential operator of order one, and its homogeneous principal symbol is
\begin{equation}
\sigma_1(\Lambda_{g,c}^T)(y_0,t_0;\eta,\omega)
=
\left(
|\eta|_{g_{\partial M}(y_0)}^2
-
\frac{\omega^2}{c(y_0,t_0)^2}
\right)^{1/2}.
\label{eq:dn-principal-symbol}
\end{equation}
In particular, \(\Lambda_{g,c}^T\) determines \(g_{\partial M}\) and \(c|_{\Sigma_T}\).
\end{lemma}

\begin{proof}
{
\medskip
\noindent\emph{Step 1: boundary normal form and the elliptic normal root.}
Choose boundary normal coordinates \((y,r)\) for \(g\) near \(y_0\), with \(r\geq0\) in \(M\), \(r=0\) on \(\partial M\), and \(\nu_g=-\partial_r\) at the boundary.
In these coordinates,
\[
\Delta_g
=
\partial_r^2+\gamma(y,r)\partial_r+\Delta_{g_r},
\]
where \(g_r\) is the induced metric on \(r=\mathrm{const}\).
Write \(x'=(y,t)\), set \(D_r=-i\partial_r\), and denote the wave operator by
\[
L=c^{-2}\partial_t^2-\Delta_g.
\]
Its homogeneous principal symbol, viewed as a polynomial in the normal covariable \(\rho\), is
\[
p_2(r,x';\rho,\eta,\omega)
=
\rho^2+|\eta|_{g_r(y)}^2-\frac{\omega^2}{c(y,r,t)^2}.
\]
To compute the boundary principal symbol, we freeze the coefficients at
\((y_0,0,t_0)\): in the principal part of \(L\), replace \(g_r(y)\) and
\(c(y,r,t)\) by their values \(g_{\partial M}(y_0)\) and \(c(y_0,t_0)\),
respectively, and test the resulting operator with constant coefficients on
\(e^{i((y-y_0)\cdot\eta+(t-t_0)\omega+r\rho)}\).
The corresponding frozen normal
polynomial is
\[
p_{2,0}(\rho)
=
\rho^2+|\eta|_{g_{\partial M}(y_0)}^2
-\frac{\omega^2}{c(y_0,t_0)^2}.
\]
Define
\[
\kappa_0
=
\left(
|\eta|_{g_{\partial M}(y_0)}^2
-
\frac{\omega^2}{c(y_0,t_0)^2}
\right)^{1/2}.
\]
Condition~\eqref{eq:elliptic-region} gives \(\kappa_0>0\), and the normal roots are
\[
\rho_\pm=\pm i\kappa_0.
\]
The root \(i\kappa_0\) produces the factor \(e^{ir(i\kappa_0)}=e^{-r\kappa_0}\), which decays as \(r\) increases into \(M\); the other root produces an exponentially increasing factor.
Thus the elliptic Dirichlet solution selects the root \(i\kappa_0\).
Since
\[
\nu_g=-\partial_r=-iD_r
\quad\text{on }r=0,
\]
the corresponding candidate principal symbol of the normal trace is
\[
-i(i\kappa_0)=\kappa_0.
\]

\medskip
\noindent\emph{Step 2: uniform microlocal factorization.}
To justify the preceding calculation with frozen coefficients, choose an interval \(I\Subset(0,T)\) containing \(t_0\), a boundary coordinate neighborhood \(V\) of \(y_0\), and a sufficiently small conic neighborhood
\[
\Gamma\Subset T^*(V\times I)\setminus0
\]
of \((y_0,t_0;\eta,\omega)\).
After decreasing the collar width \(\delta\) and shrinking \(\Gamma\), continuity and conic homogeneity give an \(\epsilon>0\) such that
\begin{equation}
|\eta|_{g_r(y)}^2-\frac{\omega^2}{c(y,r,t)^2}
\geq
\epsilon\bigl(|\eta|^2+\omega^2\bigr)
\label{eq:uniform-elliptic-cone}
\end{equation}
for \(0\leq r<\delta\) and \((y,t;\eta,\omega)\in\Gamma\).
Hence
\[
\kappa(y,r,t;\eta,\omega)
:=
\left(
|\eta|_{g_r(y)}^2-\frac{\omega^2}{c(y,r,t)^2}
\right)^{1/2}
\]
is a smooth classical symbol of order one on \(\Gamma\), and the two normal roots \(\pm i\kappa\) remain uniformly separated there.

The elliptic boundary factorization on \(\Gamma\) gives tangential classical pseudodifferential operators
\[
A_\pm(r,x',D_{x'})\in\Psi^1,
\qquad
\sigma_1(A_\pm)=\pm i\kappa,
\]
such that, microlocally on \(\Gamma\),
\begin{equation}
L
=
(D_r-A_-)(D_r-A_+)+R_\infty,
\qquad
R_\infty\in C^\infty([0,\delta);\Psi^{-\infty}).
\label{eq:microlocal-normal-factorization}
\end{equation}
Indeed, after the principal symbols \(\pm i\kappa\) have been fixed, comparison of terms of successively lower homogeneity determines the remaining symbol coefficients recursively; Borel summation then gives \(A_\pm\) with a smoothing remainder.
The term \(\gamma\partial_r\) and the derivatives of \(c\) and \(g_r\) enter only in these lower-order coefficients.
This is the scalar elliptic factorization underlying the boundary-symbol argument in \cite[Lemma~2.1]{eskin2017general}; the general symbolic constructions and the corresponding operators on the boundary are given in \cite[Sections~18.1 and~20.1]{hormander1985analysis3} and \cite[Sections~1.2 and~1.4]{grubb1996functional}.

\medskip
\noindent\emph{Step 3: the Poisson parametrix and the causal finite-time map.}
Let \(Q_0,Q_1\in\Psi^0(\Sigma_T)\) be properly supported cutoffs whose kernels lie in \((V\times I)^2\), such that
\[
\operatorname{WF}'(Q_j)\subset\Gamma
\]
and \(Q_1\) equals the identity microlocally on \(\operatorname{WF}'(Q_0)\).
The decaying factor in \eqref{eq:microlocal-normal-factorization} determines a Poisson parametrix \(\mathcal P\) satisfying
\[
(D_r-A_+)\mathcal P\equiv0,
\qquad
\mathcal P|_{r=0}\equiv Q_0
\quad\bmod C^\infty.
\]
Its leading normal factor is
\[
\exp\!\left(
-\int_0^r
\kappa(y,s,t;\eta,\omega)\,\mathrm{d}s
\right).
\]
Choose \(\mathcal P\) properly supported in the tangential variables over \(V\times I\).
Then \(Q_0f\in\mathcal D_T\) and \(\mathcal Pf=0\) near \(t=0\) for every \(f\in C^\infty(\Sigma_T)\).
Let \(\gamma_0u=u|_{\Sigma_T}\) and
\(\gamma_1u=\partial_{\nu_g}u|_{\Sigma_T}\).
Solving the transport
equations to all orders and using
\eqref{eq:microlocal-normal-factorization}, followed by multiplication by a
cutoff in \(r\) that equals one near \(r=0\), gives the remainders
\[
E_{\mathrm{int}}:=L\mathcal P,
\qquad
E_0:=\gamma_0\mathcal P-Q_0.
\]
They are smoothing in the tangential variables: for every \(N,M\geq0\),
\begin{equation}
E_{\mathrm{int}}:H^{-N}(\Sigma_T)\longrightarrow H^M(Q_T),
\qquad
E_0:H^{-N}(\Sigma_T)\longrightarrow H^M(\Sigma_T)
\label{eq:poisson-smoothing-remainders}
\end{equation}
continuously, with ranges supported away from the initial time face.
For the
collar cutoff error this follows from the factor
\(\exp(-\int_0^r\kappa\,\mathrm{d}s)\): on the support of the derivative of
the cutoff, \(r\geq r_0>0\), and
\eqref{eq:uniform-elliptic-cone} gives exponential decay
\(e^{-c r_0|(\eta,\omega)|}\), which belongs to every negative symbol class.

For \(f\in C^\infty(\Sigma_T)\), let \(r_f\) solve the causal correction
problem
\begin{equation}
\begin{cases}
Lr_f=-E_{\mathrm{int}}f&\text{in }Q_T,\\
\gamma_0r_f=-E_0f&\text{on }\Sigma_T,\\
r_f|_{t=0}=\partial_tr_f|_{t=0}=0&\text{in }M.
\end{cases}
\label{eq:causal-parametrix-correction}
\end{equation}
Because the ranges in \eqref{eq:poisson-smoothing-remainders} are supported away from \(t=0\), both the interior forcing and the boundary correction vanish near the initial time and agree with the zero initial data.
Energy estimates followed by hyperbolic boundary regularity, applied at an
arbitrary Sobolev order to the smoothing data in
\eqref{eq:poisson-smoothing-remainders}, give, for every \(N,M\geq0\),
\begin{equation}
\|\gamma_1r_f\|_{H^M(\Sigma_T)}
\leq C_{N,M}\|f\|_{H^{-N}(\Sigma_T)}.
\label{eq:smoothing-normal-correction}
\end{equation}
Thus \(f\mapsto\gamma_1r_f\) is smoothing.
Moreover,
\(\mathcal Pf+r_f\) has zero initial data, solves \(L(\mathcal Pf+r_f)=0\), and has boundary
value \(Q_0f\); uniqueness in
Proposition~\ref{prop:forward-well-posedness-causality} therefore gives
\[
u^{Q_0f}=\mathcal Pf+r_f.
\]
Finally, \((D_r-A_+)\mathcal P\equiv0\) and \(\nu_g=-iD_r\) imply
\[
\gamma_1\mathcal P=-iA_+(0)Q_0\quad\bmod\Psi^{-\infty}.
\]
Combining this identity with \eqref{eq:smoothing-normal-correction} proves
\begin{equation}
Q_1\Lambda_{g,c}^TQ_0
=
-iQ_1A_+(0)Q_0
\quad\bmod\Psi^{-\infty}.
\label{eq:localized-dn-factor}
\end{equation}
The Poisson and trace mapping properties used in
\eqref{eq:poisson-smoothing-remainders}--\eqref{eq:smoothing-normal-correction}
follow from \cite[Sections~1.2 and~1.4]{grubb1996functional}.
Regularity at an arbitrary Sobolev order for the hyperbolic mixed problem that passes
from the smoothing data in \eqref{eq:poisson-smoothing-remainders} to the
estimate for the normal trace \eqref{eq:smoothing-normal-correction} is supplied by
\cite[Theorem~24.1.1]{hormander1985analysis3}.
The comparison between a local
parametrix and an actual global solution modulo smoothing errors is described
in \cite[Section~5]{stefanov2018inverse}; that section invokes
the same theorem for the hyperbolic mixed problem and the true solution.
The cutoffs are supported in \(I\Subset(0,T)\), so the initial and terminal
time faces do not enter the construction.
Moreover,
\eqref{eq:uniform-elliptic-cone} separates \(\Gamma\) from the hyperbolic and
glancing regions.
Because \(\Gamma\) is separated from the hyperbolic and glancing regions, the localized input has no real boundary bicharacteristic issuing from it.
Hence the Fourier integral components associated with propagation do not contribute to \(Q_1\Lambda_{g,c}^TQ_0\), leaving only the pseudodifferential operator in \eqref{eq:localized-dn-factor}; see \cite[Theorem~4.2]{stefanov2018inverse} for the propagated part represented by Fourier integral operators.

Taking principal symbols in \eqref{eq:localized-dn-factor} yields
\[
\sigma_1(\Lambda_{g,c}^T)
=
-i\sigma_1(A_+)
=
\kappa(y,0,t;\eta,\omega).
\]
At \((y_0,t_0;\eta,\omega)\), this is precisely \eqref{eq:dn-principal-symbol}.

\medskip
\noindent\emph{Step 4: recovery of the boundary metric and wave speed.}
First take \(\omega=0\).
Then every \(\eta\neq0\) lies in the elliptic region, and
\[
\sigma_1(\Lambda_{g,c}^T)(y_0,t_0;\eta,0)^2
=
|\eta|_{g_{\partial M}(y_0)}^2
=
g_{\partial M}^{\alpha\beta}(y_0)\eta_\alpha\eta_\beta.
\]
Thus the principal symbol determines the inverse boundary metric, and hence \(g_{\partial M}(y_0)\).
Once this metric is known, choose \(\omega\neq0\) and \(\eta\) satisfying \eqref{eq:elliptic-region}.
Squaring and rearranging \eqref{eq:dn-principal-symbol} gives
\[
c(y_0,t_0)^2
=
\frac{\omega^2}
{|\eta|_{g_{\partial M}(y_0)}^2
-\sigma_1(\Lambda_{g,c}^T)(y_0,t_0;\eta,\omega)^2}.
\]
Positivity determines \(c(y_0,t_0)\) uniquely.
Since \((y_0,t_0)\in\Sigma_T\) was arbitrary, the map determines \(g_{\partial M}\) and \(c|_{\Sigma_T}\).

The proof is complete.
}
\end{proof}

Under the conjugacy \eqref{eq:boundary-identification}, the natural transformation law of the principal symbol gives
\begin{equation}
\varphi^*(g_{2,\partial M_2})
=
g_{1,\partial M_1},
\qquad
c_2(\varphi(y),t)=c_1(y,t),
\qquad
(y,t)\in\partial M_1\times(0,T).
\label{eq:boundary-metric-speed-recovery}
\end{equation}

Smoothness extends the second identity in \eqref{eq:boundary-metric-speed-recovery} to \(t=0\) and \(t=T\) by one-sided limits.
For \(c\in\mathcal C^T(M)\), the normalization \(b(0)=1\) gives \(c(y,0)=a(y)\), and hence, for every \(y_0\in\partial M\),
\begin{equation}
b(t)
=
\frac{c(y_0,t)}{c(y_0,0)},
\qquad
0\leq t\leq T.
\label{eq:recover-b}
\end{equation}
This boundary recovery supplies the temporal information used below.

\begin{lemma}
\label{lem:temporal-factor}
Let \(c_j\in\mathcal C^T(M_j)\), \(j=1,2\), have temporal factors \(b_j\), and suppose that their boundary maps satisfy \eqref{eq:boundary-identification}.
Then
\[
b_1=b_2
\qquad\text{on }[0,T].
\]
\end{lemma}

\begin{proof}
Fix any \(y_0\in\partial M_1\).
By smoothness, the second identity in \eqref{eq:boundary-metric-speed-recovery} extends to \(t=0\) and \(t=T\), so \eqref{eq:recover-b} gives, for \(0\leq t\leq T\),
\begin{equation}
b_1(t)
=
\frac{c_1(y_0,t)}{c_1(y_0,0)}
=
\frac{c_2(\varphi(y_0),t)}{c_2(\varphi(y_0),0)}
=
b_2(t).\qedhere
\label{eq:common-b}
\end{equation}
\end{proof}

Denote the recovered common temporal factor by \(b\).
It determines \(B\), its inverse \(\beta\), \(S_T\), and all temporal coefficients used below.

\subsection{Reparametrization by effective time and scalar gauge}

For \(c(x,t)=a(x)b(t)\) as in \eqref{eq:main-separable-speed}, multiplying the interior equation in \eqref{eq:main-forward-problem} by \(c(x,t)^2\) gives
\[
u_{tt}-a(x)^2b(t)^2\Delta_g u=0.
\]
Thus \(b(t)^2\) is the time-dependent scalar multiplying the spatial principal operator.
The recovered temporal factor defines the propagation clock
\begin{equation}
B(t)
=
\int_0^t b(r)\,\mathrm{d}r,
\qquad
S_T=B(T).
\label{eq:effective-time}
\end{equation}
Since \(b>0\), the map \(B:[0,T]\to[0,S_T]\) is a smooth increasing diffeomorphism.
Let \(\beta=B^{-1}:[0,S_T]\to[0,T]\).
Thus \(\beta(s)\) is the physical time corresponding to the effective time \(s\), and
\[
B(\beta(s))=s,
\qquad
\beta(B(t))=t.
\]

The reparametrization \(s=B(t)\) removes \(b(t)^2\) from the spatial principal part and leaves only a first-order temporal term determined by \(b\).
The following lemma records the transformed equation.
\begin{lemma}
\label{lem:effective-time}
Let \(c\in\mathcal C^T(M)\) have the normalized factors \(a\) and \(b\) from \eqref{eq:main-separable-speed}, let \(u\) solve \eqref{eq:main-forward-problem}, and set
\[
v(x,s)=u(x,\beta(s)).
\]
Then
\begin{equation}
v_{ss}+p(s)v_s-a(x)^2\Delta_g v=0,
\qquad
p(s)
=
\frac{b'(\beta(s))}{b(\beta(s))^2}.
\label{eq:v-equation}
\end{equation}
The zero initial data are preserved.
\end{lemma}

\begin{proof}
Since \(\beta=B^{-1}\), the definition of \(v\) is equivalent to \(u(x,t)=v(x,B(t))\).
Since \(B'(t)=b(t)\), the chain rule gives \(u_t(x,t)=b(t)v_s(x,B(t))\).
Differentiating once more gives
\[
u_{tt}(x,t)
=
b(t)^2v_{ss}(x,B(t))+b'(t)v_s(x,B(t)).
\]
Because \(B\) is independent of \(x\), the spatial derivatives are unchanged:
\[
\Delta_g u(x,t)
=
\Delta_g v(x,B(t)).
\]
Substitution into \(u_{tt}-a(x)^2b(t)^2\Delta_g u=0\), followed by division by \(b(t)^2\) and the identity \(t=\beta(s)\), gives
\[
v_{ss}
+
\frac{b'(\beta(s))}{b(\beta(s))^2}v_s
-
a(x)^2\Delta_gv
=0,
\]
which is \eqref{eq:v-equation}.
Finally, \(B(0)=\beta(0)=0\), and the identities above give
\[
v(x,0)=u(x,0)=0,
\qquad
v_s(x,0)=b(0)^{-1}u_t(x,0)=0.
\]
Thus both zero initial conditions are preserved.
\end{proof}

We remove the remaining first-order term \(p(s)v_s\) by introducing the scalar gauge
\begin{equation}
\vartheta(s)
=
\exp\left(\frac12\int_0^s p(\sigma)\,\mathrm{d}\sigma\right),
\qquad
w(x,s)=\vartheta(s)v(x,s).
\label{eq:vartheta-definition}
\end{equation}
Define
\begin{equation}
q(s)
=
-\frac12p'(s)-\frac14p(s)^2.
\label{eq:q-p}
\end{equation}
On boundary functions, set
\begin{equation}
(U_b^{S_T}f)(y,s)
=
\vartheta(s)f(y,\beta(s)),
\qquad
0<s<S_T.
\label{eq:U-definition}
\end{equation}
The next lemma identifies the transformed solution and its boundary response.

\begin{lemma}
\label{lem:effective-time-rescaling}
Under the assumptions of Lemma~\ref{lem:effective-time}, the function \(w\) satisfies
\begin{equation}
w_{ss}-a(x)^2\Delta_g w+q(s)w=0,
\qquad
w(\cdot,0)=w_s(\cdot,0)=0.
\label{eq:w-equation}
\end{equation}
Moreover, \(U_b^{S_T}:\mathcal D_T\to\mathcal D_{S_T}\) is invertible and
\begin{equation}
\Lambda_{g,a,q}^{S_T}
=
U_b^{S_T}\Lambda_{g,c}^T(U_b^{S_T})^{-1},
\label{eq:potential-response}
\end{equation}
where \(\Lambda_{g,a,q}^{S_T}\) is the Dirichlet-to-Neumann map for \eqref{eq:w-equation} on \((0,S_T)\).
\end{lemma}

\begin{proof}
Differentiating \(B(\beta(s))=s\) gives \(\beta'(s)=b(\beta(s))^{-1}\), and hence
\[
\frac{\mathrm{d}}{\mathrm{d}s}\log b(\beta(s))
=
\frac{b'(\beta(s))}{b(\beta(s))^2}
=
p(s).
\]
Using \(\beta(0)=0\) and \(b(0)=1\), integration over \([0,s]\) gives \(\vartheta(s)=\sqrt{b(\beta(s))}\) and \(\vartheta'/\vartheta=p/2\).

Since \(v=\vartheta^{-1}w\),
\[
v_s
=
\vartheta^{-1}\left(w_s-\frac12pw\right)
\]
and
\[
v_{ss}
=
\vartheta^{-1}
\left(
w_{ss}-pw_s+\left(\frac14p^2-\frac12p'\right)w
\right).
\]
It follows that
\[
v_{ss}+pv_s
=
\vartheta^{-1}
\left(
w_{ss}-\frac12p'w-\frac14p^2w
\right).
\]
Because \(\vartheta\) is independent of \(x\), one also has \(\Delta_gv=\vartheta^{-1}\Delta_gw\).
Substituting these identities into \eqref{eq:v-equation} gives the differential equation in \eqref{eq:w-equation}.
The initial conditions follow from
\[
w(\cdot,0)=\vartheta(0)v(\cdot,0)=0,
\qquad
w_s(\cdot,0)=\vartheta'(0)v(\cdot,0)+\vartheta(0)v_s(\cdot,0)=0.
\]

To express \(q\) in terms of \(b\), write \(t=\beta(s)\).
Since \(\mathrm{d}t/\mathrm{d}s=b(t)^{-1}\), differentiating \(p(s)=b'(t)/b(t)^2\) gives \(p'(s)=b''(t)/b(t)^3-2b'(t)^2/b(t)^4\).
Therefore,
\begin{equation}
q(s)
=
-\frac{b''(t)}{2b(t)^3}
+
\frac{3b'(t)^2}{4b(t)^4},
\qquad
t=\beta(s).
\label{eq:q-b}
\end{equation}
Thus \(q|_{[0,S_T]}\) is determined by \(b|_{[0,T]}\).

Let \(f\in\mathcal D_T\), and let \(u\) be the solution with boundary value \(f\).
The boundary value of \(w\) is
\[
w(y,s)
=
\vartheta(s)u(y,\beta(s))
=
\vartheta(s)f(y,\beta(s))
=
(U_b^{S_T}f)(y,s).
\]
Since \(\vartheta\) is independent of \(x\), its Neumann trace satisfies
\[
\begin{aligned}
\partial_{\nu_g}w(y,s)
&=
\vartheta(s)\partial_{\nu_g}u(y,\beta(s))
\\
&=
\vartheta(s)(\Lambda_{g,c}^Tf)(y,\beta(s))
\\
&=
(U_b^{S_T}\Lambda_{g,c}^Tf)(y,s).
\end{aligned}
\]
Since \(w\) has Dirichlet data \(U_b^{S_T}f\), the preceding identity gives
\[
\Lambda_{g,a,q}^{S_T}(U_b^{S_T}f)
=
U_b^{S_T}(\Lambda_{g,c}^Tf).
\]
Since this holds for every \(f\in\mathcal D_T\),
\begin{equation}
\Lambda_{g,a,q}^{S_T} U_b^{S_T}
=
U_b^{S_T}\Lambda_{g,c}^T.
\label{eq:first-conjugacy}
\end{equation}

Finally, \(\vartheta(B(t))=\sqrt{b(t)}\), so
\begin{equation}
((U_b^{S_T})^{-1}\psi)(y,t)
=
b(t)^{-1/2}\psi(y,B(t)).
\label{eq:U-inverse}
\end{equation}
The smooth maps \(B\) and \(\beta\), together with the positive factor \(\vartheta\), show that \(U_b^{S_T}\) and its inverse map \(\mathcal D_T\) and \(\mathcal D_{S_T}\) onto each other.
Composing \eqref{eq:first-conjugacy} on the right with \((U_b^{S_T})^{-1}\) gives \eqref{eq:potential-response}.

The proof is complete.
\end{proof}

Since \(U_b^{S_T}\) is determined by \(b\), equation~\eqref{eq:potential-response} determines \(\Lambda_{g,a,q}^{S_T}\) from \((\Lambda_{g,c}^T,b)\).

\subsection{Causal transmutation and the exact response correspondence}

Let \(q\in C^\infty([0,S_T])\) be the temporal potential determined by the recovered temporal factor.
We construct an invertible causal operator that intertwines \(\partial_s^2+q(s)\) with \(\partial_s^2\) and preserves both the zero initial data and the boundary source space.
This operator yields the exact finite-time response correspondence.

\begin{definition}
\label{def:volterra-operator}
Given a continuous kernel \(K\) on \(\{(s,r):0\leq r\leq s\leq S_T\}\), define
\[
(\mathcal V_Kz)(s)
=
\int_0^s K(s,r)z(r)\,\mathrm{d}r.
\]
The operator \(\mathcal V_K\) is called a \emph{Volterra integral operator}.
The triangular restriction \(0\leq r\leq s\) means that \((\mathcal V_Kz)(s)\) depends only on \(z|_{[0,s]}\), so the operator is causal in time.
The equation \((I+\mathcal V_K)z=f\) is a Volterra integral equation of the second kind.
\end{definition}

For background on Volterra integral equations and transformation operators, see
\cite{gripenberg1990volterra,levitan1987inverse,marchenko2011sturm,campos2017standard}.

To obtain the intertwining identity between \(\partial_s^2\) and \(\partial_s^2+q(s)\), the kernel \(K\) in Definition~\ref{def:volterra-operator} must satisfy a hyperbolic Goursat problem.
The following lemma shows that this problem has a unique smooth solution.

\begin{lemma}
\label{lem:transmutation-kernel}
Let \(q\in C^\infty([0,S_T])\) be the temporal potential defined in \eqref{eq:q-p}.
Then there exists a unique kernel \(K\) defining the Volterra integral operator \(\mathcal V_K\) in Definition~\ref{def:volterra-operator}.
This kernel satisfies
\[
K\in C^\infty\bigl(\{(s,r):0\leq r\leq s\leq S_T\}\bigr)
\]
and
\begin{align}
&K_{ss}(s,r)-K_{rr}(s,r)+q(s)K(s,r)=0,
\label{eq:kernel-pde}\\
&K(s,0)=0,
\label{eq:kernel-axis}\\
&K(s,s)=-\frac12\int_0^s q(\sigma)\,\mathrm{d}\sigma.
\label{eq:kernel-diagonal}
\end{align}
\end{lemma}

\begin{proof}
Write
\[
\Delta_{S_T}=\{(s,r):0\leq r\leq s\leq S_T\}.
\]
We first transform the differential problem on \(\Delta_{S_T}\) into an equivalent triangular integral equation and then solve that equation by a contraction argument.

\medskip
\noindent\emph{Step 1: transformation to characteristic coordinates.}
Introduce the characteristic variables
\[
\xi=s+r,
\qquad
\eta=s-r,
\qquad
H(\xi,\eta)
=
K\left(\frac{\xi+\eta}{2},\frac{\xi-\eta}{2}\right).
\]
The inverse relations are
\[
s=\frac{\xi+\eta}{2},
\qquad
r=\frac{\xi-\eta}{2}.
\]
Thus the change of variables maps \(\Delta_{S_T}\) bijectively onto
\[
D_{S_T}
=
\{(\xi,\eta):0\leq\eta\leq\xi,\ \xi+\eta\leq2S_T\}.
\]
The sides \(r=s\) and \(r=0\) correspond respectively to \(\eta=0\) and \(\xi=\eta\).
Since \(K(s,r)=H(s+r,s-r)\), the chain rule gives
\begin{align*}
K_s&=H_\xi+H_\eta,
&
K_r&=H_\xi-H_\eta,\\
K_{ss}&=H_{\xi\xi}+2H_{\xi\eta}+H_{\eta\eta},
&
K_{rr}&=H_{\xi\xi}-2H_{\xi\eta}+H_{\eta\eta},
\end{align*}
where all derivatives of \(H\) are evaluated at \((\xi,\eta)=(s+r,s-r)\).
Subtracting the last two identities gives \(K_{ss}-K_{rr}=4H_{\xi\eta}\).
Consequently, equations \eqref{eq:kernel-pde}--\eqref{eq:kernel-diagonal} are equivalent to
\begin{align}
4H_{\xi\eta}(\xi,\eta)
+
q\left(\frac{\xi+\eta}{2}\right)H(\xi,\eta)&=0,
\label{eq:kernel-characteristic-pde}\\
H(\xi,0)=-\frac12\int_0^{\xi/2}q(\sigma)\,\mathrm{d}\sigma,
\qquad
H(\xi,\xi)&=0.
\notag
\end{align}

\medskip
\noindent\emph{Step 2: the equivalent triangular integral equation.}
Fix \((\xi,\eta)\in D_{S_T}\).
The rectangle
\[
R_{\xi,\eta}=[\eta,\xi]\times[0,\eta]
\]
is contained in \(D_{S_T}\).
Indeed, if \((\xi',\eta')\in R_{\xi,\eta}\), then \(0\leq\eta'\leq\eta\leq\xi'\) and \(\xi'+\eta'\leq\xi+\eta\leq2S_T\).
Integrating \eqref{eq:kernel-characteristic-pde} over \(R_{\xi,\eta}\) gives
\[
\int_\eta^\xi\int_0^\eta
H_{\xi\eta}(\xi',\eta')\,\mathrm{d}\eta'\,\mathrm{d}\xi'
=
-\frac14
\int_\eta^\xi\int_0^\eta
q\left(\frac{\xi'+\eta'}{2}\right)
H(\xi',\eta')\,\mathrm{d}\eta'\,\mathrm{d}\xi'.
\]
For the left hand side, first integrate with respect to \(\eta'\), and then with respect to \(\xi'\):
\[
\begin{aligned}
\int_\eta^\xi\int_0^\eta
H_{\xi\eta}(\xi',\eta')\,\mathrm{d}\eta'\,\mathrm{d}\xi'
&=
\int_\eta^\xi
\bigl[H_\xi(\xi',\eta)-H_\xi(\xi',0)\bigr] \,\mathrm{d}\xi'\\
&=
\bigl[H(\xi,\eta)-H(\eta,\eta)\bigr]
-
\bigl[H(\xi,0)-H(\eta,0)\bigr].
\end{aligned}
\]
Combining the last two identities yields
\[
\begin{aligned}
&H(\xi,\eta)-H(\eta,\eta)-H(\xi,0)+H(\eta,0)\\
&\qquad
=
-\frac14
\int_\eta^\xi\int_0^\eta
q\left(\frac{\xi'+\eta'}{2}\right)
H(\xi',\eta')
\,\mathrm{d}\eta'\,\mathrm{d}\xi'.
\end{aligned}
\]
Substituting the two boundary values and solving for \(H(\xi,\eta)\) yields the triangular integral equation
\begin{equation}
H(\xi,\eta)
=
-\frac12\int_{\eta/2}^{\xi/2}q(\sigma)\,\mathrm{d}\sigma
-
\frac14
\int_\eta^\xi\int_0^\eta
q\left(\frac{\xi'+\eta'}{2}\right)
H(\xi',\eta')
\,\mathrm{d}\eta'\,\mathrm{d}\xi'.
\label{eq:H-volterra}
\end{equation}
For smooth \(H\), this integral equation is equivalent to the transformed differential problem.
Indeed, setting \(\eta=0\) in \eqref{eq:H-volterra} gives the prescribed value of \(H(\xi,0)\), while setting \(\xi=\eta\) gives \(H(\xi,\xi)=0\).
Applying \(\partial_\eta\partial_\xi\) to \eqref{eq:H-volterra} recovers the transformed differential equation.

\medskip
\noindent\emph{Step 3: existence and uniqueness of a continuous solution.}
Let \(\mathfrak F:C(D_{S_T})\to C(D_{S_T})\) denote the map defined by the right hand side of \eqref{eq:H-volterra}.
For \(\lambda>0\), equip \(C(D_{S_T})\) with the weighted norm
\[
\|H\|_\lambda
=
\sup_{(\xi,\eta)\in D_{S_T}}
e^{-\lambda(\xi+\eta)}|H(\xi,\eta)|.
\]
Since \(0\leq\xi+\eta\leq2S_T\) on \(D_{S_T}\),
\[
e^{-2\lambda S_T}\|H\|_\infty
\leq
\|H\|_\lambda
\leq
\|H\|_\infty.
\]
Thus this weighted norm is equivalent to the uniform norm, and \(C(D_{S_T})\) is complete in \(\|\cdot\|_\lambda\).
Put \(C_q=\|q\|_{L^\infty(0,S_T)}\).
For \(H_1,H_2\in C(D_{S_T})\), the fixed first term in \eqref{eq:H-volterra} cancels, and
\begin{align*}
&e^{-\lambda(\xi+\eta)}
|(\mathfrak F H_1)(\xi,\eta)-(\mathfrak F H_2)(\xi,\eta)|
\\
&\leq
\frac{C_q}{4}\|H_1-H_2\|_\lambda
\int_\eta^\xi e^{-\lambda(\xi-\xi')}\,\mathrm{d}\xi'
\int_0^\eta e^{-\lambda(\eta-\eta')}\,\mathrm{d}\eta'\\
&\leq
\frac{C_q}{4\lambda^2}\|H_1-H_2\|_\lambda.
\end{align*}
Taking the supremum over \(D_{S_T}\) gives
\[
\|\mathfrak F H_1-\mathfrak F H_2\|_\lambda
\leq
\frac{C_q}{4\lambda^2}\|H_1-H_2\|_\lambda.
\]
Choose \(\lambda>\sqrt{C_q}/2\).
Then \(C_q/(4\lambda^2)<1\), so \(\mathfrak F\) is a contraction.
The Banach fixed point theorem therefore gives a unique \(H\in C(D_{S_T})\) satisfying \eqref{eq:H-volterra}.

\medskip
\noindent\emph{Step 4: smoothness and return to the original variables.}
Let \(H\) be this fixed point.
Since \(q\) and \(H\) are continuous, differentiating the right hand side of \eqref{eq:H-volterra} with respect to its integration limits shows that \(H\in C^1(D_{S_T})\).
If \(H\in C^m(D_{S_T})\), then repeated use of the Leibniz rule and \(q\in C^\infty([0,S_T])\) shows that the right hand side of \eqref{eq:H-volterra} belongs to \(C^{m+1}(D_{S_T})\).
Induction gives \(H\in C^\infty(D_{S_T})\).
By Step 2, this smooth fixed point solves the transformed differential problem.

Define
\[
K(s,r)=H(s+r,s-r)
\]
on \(\Delta_{S_T}\).
Step 1 shows that \(K\) is smooth up to the boundary and satisfies \eqref{eq:kernel-pde}--\eqref{eq:kernel-diagonal}.
If another smooth solution existed, the same coordinate change would produce a second continuous fixed point of \(\mathfrak F\), contradicting the uniqueness in Step 3.
Thus \(K\) is the unique smooth solution.

The proof is complete.
\end{proof}

Let \(K\) be the unique kernel supplied by Lemma~\ref{lem:transmutation-kernel}, and define the causal transmutation operator \(T_q^{S_T}=I+\mathcal V_K\) by
\begin{equation}
(T_q^{S_T}z)(s)
=
z(s)+\int_0^sK(s,r)z(r)\,\mathrm{d}r.
\label{eq:Tq-definition}
\end{equation}

\begin{lemma}
\label{lem:intertwining}
If \(z\in C^2([0,S_T])\) satisfies \(z(0)=z'(0)=0\), then
\begin{equation}
(\partial_s^2+q(s))T_q^{S_T}z
=
T_q^{S_T}(\partial_s^2z).
\label{eq:one-dimensional-intertwining}
\end{equation}
The operator \(T_q^{S_T}\) is invertible on \(C^m([0,S_T])\) for every integer \(m\geq0\).
For every \(m\geq1\), it restricts to an automorphism of
\[
\left\{
z\in C^m([0,S_T])
:
z(0)=z'(0)=0
\right\}.
\]
Its inverse again has the same causal triangular integral form.
Moreover, \(T_q^{S_T}\) and its inverse preserve vanishing on an initial time interval.
When applied pointwise in the boundary variable, they therefore define mutually inverse linear maps
\[
T_q^{S_T},
\ (T_q^{S_T})^{-1}
:
\mathcal D_{S_T}
\longrightarrow
\mathcal D_{S_T}.
\]
\end{lemma}

\begin{proof}
Write \(\mathcal V=\mathcal V_K\) and \(\mathcal T=T_q^{S_T}=I+\mathcal V\).
We prove the assertions in four steps.

\medskip
\noindent\emph{Step 1: the intertwining identity.}
Let \(z\in C^2([0,S_T])\) satisfy \(z(0)=z'(0)=0\).
Applying Leibniz' rule to \((\mathcal Vz)(s)=\int_0^sK(s,r)z(r)\,\mathrm{d}r\) gives
\begin{align}
(\mathcal Vz)'(s)
&=
K(s,s)z(s)
+
\int_0^sK_s(s,r)z(r)\,\mathrm{d}r,
\label{eq:Vz-first}\\
(\mathcal Vz)''(s)
&=
(2K_s+K_r)(s,s)z(s)
+
K(s,s)z'(s)
+
\int_0^sK_{ss}(s,r)z(r)\,\mathrm{d}r.
\label{eq:Vz-second}
\end{align}
For comparison with \(\mathcal V(z'')\), integrate twice by parts in \(r\).
The conditions \(z(0)=z'(0)=0\) eliminate the terms at \(r=0\), and hence
\begin{equation}
\int_0^sK(s,r)z''(r)\,\mathrm{d}r
=
K(s,s)z'(s)
-
K_r(s,s)z(s)
+
\int_0^sK_{rr}(s,r)z(r)\,\mathrm{d}r.
\label{eq:twice-by-parts}
\end{equation}
Since \(\mathcal T=I+\mathcal V\), equations \eqref{eq:Vz-second} and \eqref{eq:twice-by-parts} give
\begin{align*}
(\partial_s^2+q)\mathcal Tz-\mathcal Tz''
={}&
\bigl[2(K_s+K_r)(s,s)+q(s)\bigr]z(s)\\
&+
\int_0^s
\bigl(K_{ss}-K_{rr}+q(s)K\bigr)(s,r)z(r)\,\mathrm{d}r.
\end{align*}
The kernel identities supplied by Lemma~\ref{lem:transmutation-kernel} now give both cancellations.
Indeed, the integral term vanishes by \eqref{eq:kernel-pde}, while differentiating \eqref{eq:kernel-diagonal} along the diagonal gives
\[
(K_s+K_r)(s,s)
=
-\frac12q(s),
\]
and therefore
\[
2(K_s+K_r)(s,s)+q(s)=0.
\]
Thus the pointwise term also vanishes, and
\[
(\partial_s^2+q)\mathcal Tz=\mathcal Tz''.
\]
This proves \eqref{eq:one-dimensional-intertwining}.

\medskip
\noindent\emph{Step 2: invertibility on \(C([0,S_T])\) and the form of the inverse.}
Let \(C_K=\|K\|_{L^\infty}\).
For \(m\geq1\), the \(m\)-fold iterate of \(\mathcal V\) is integrated over the simplex
\[
0\leq r_m\leq r_{m-1}\leq\cdots\leq r_1\leq s.
\]
Since this simplex has volume \(s^m/m!\), bounding the \(m\) kernel factors by \(C_K\) gives
\[
\|\mathcal V^m\|_{C\to C}
\leq
\frac{(C_KS_T)^m}{m!}.
\]
Here \(\|\cdot\|_{C\to C}\) denotes the operator norm on \(C([0,S_T])\).
Hence \(\sum_{m=0}^\infty(-\mathcal V)^m\) converges in operator norm.
Let
\[
P_N=\sum_{m=0}^N(-\mathcal V)^m.
\]
Then the finite identity for a geometric series gives
\[
(I+\mathcal V)P_N
=
P_N(I+\mathcal V)
=
I+(-1)^N\mathcal V^{N+1}.
\]
Since \(\|\mathcal V^{N+1}\|_{C\to C}\to0\), passing to the limit yields
\[
\mathcal T^{-1}
=
\sum_{m=0}^\infty(-\mathcal V)^m.
\]

To identify the form of the inverse, define \(K_1=K\) and
\[
K_{m+1}(s,r)
=
\int_r^sK(s,\rho)K_m(\rho,r)\,\mathrm{d}\rho.
\]
Fubini's theorem gives \(\mathcal V^m=\mathcal V_{K_m}\), and induction gives
\[
|K_m(s,r)|
\leq
C_K^m\frac{(s-r)^{m-1}}{(m-1)!}.
\]
Thus
\[
L(s,r)
=
\sum_{m=1}^\infty(-1)^mK_m(s,r)
\]
converges uniformly on \(0\leq r\leq s\leq S_T\), and
\[
\mathcal T^{-1}=I+\mathcal V_L.
\]
Hence the inverse has the same causal triangular integral form as \(\mathcal T\).

\medskip
\noindent\emph{Step 3: invertibility on \(C^m([0,S_T])\) and preservation of the initial data.}
Since \(K\) is smooth, \(\mathcal T\) maps \(C^m([0,S_T])\) continuously into itself for every \(m\geq0\).
Let \(f\in C^m([0,S_T])\), and let \(z=\mathcal T^{-1}f\in C([0,S_T])\), whose existence follows from Step 2.
The equation \(\mathcal Tz=f\) is
\[
z(s)+\int_0^sK(s,r)z(r)\,\mathrm{d}r=f(s).
\]
If \(m\geq1\), differentiating once gives
\[
z'(s)
=
f'(s)-K(s,s)z(s)-\int_0^sK_s(s,r)z(r)\,\mathrm{d}r.
\]
Thus \(f\in C^1\) implies \(z\in C^1\).
Repeated differentiation of this identity shows inductively that \(f\in C^m\) implies \(z\in C^m\).
Therefore \(\mathcal T:C^m([0,S_T])\to C^m([0,S_T])\) is bijective; since it is bounded, the bounded inverse theorem shows that \(\mathcal T^{-1}\) is bounded on \(C^m([0,S_T])\).

It remains to check the initial conditions.
By \eqref{eq:Tq-definition}, \eqref{eq:Vz-first}, and \(K(0,0)=0\) from \eqref{eq:kernel-axis},
\[
(\mathcal Tz)(0)=z(0),
\qquad
(\mathcal Tz)'(0)=z'(0).
\]
Consequently, if \(f=\mathcal Tz\), then \(f(0)=f'(0)=0\) if and only if \(z(0)=z'(0)=0\).
This proves the asserted automorphism property for every \(m\geq1\).

\medskip
\noindent\emph{Step 4: preservation of vanishing on an initial interval.}
Let \(0<\tau\leq S_T\).
If \(z=0\) on \([0,\tau]\), then the triangular formula for \(\mathcal Tz\) immediately gives \(\mathcal Tz=0\) on the same interval.
Conversely, suppose that \(f=\mathcal Tz=0\) on \([0,\tau]\).
For \(0\leq s\leq\tau\),
\[
z(s)
=
-\int_0^sK(s,r)z(r)\,\mathrm{d}r,
\qquad
|z(s)|
\leq
C_K\int_0^s|z(r)|\,\mathrm{d}r.
\]
Gronwall's inequality gives \(z=0\) on \([0,\tau]\).
Thus both \(\mathcal T\) and \(\mathcal T^{-1}\) preserve vanishing on every initial time interval.

Finally, apply these operators pointwise in the boundary variable.
Step 3 preserves smoothness, while Step 4 preserves the defining condition that functions in \(\mathcal D_{S_T}\) vanish on an initial time interval.
Therefore \(T_q^{S_T}\) and \((T_q^{S_T})^{-1}\) define mutually inverse linear maps from \(\mathcal D_{S_T}\) onto itself.

The proof is complete.
\end{proof}

We use the same notation \(T_q^{S_T}\) for its pointwise action on functions of \((x,s)\).
Because its kernel depends only on \((s,r)\), it commutes with every spatial differential operator \(A_x\) whose coefficients are independent of \(s\).
Combining this fact with Lemma~\ref{lem:intertwining} gives
\begin{equation}
(\partial_s^2-a(x)^2\Delta_g+q(s))T_q^{S_T}z
=
T_q^{S_T}(\partial_s^2-a(x)^2\Delta_g)z.
\label{eq:full-intertwining}
\end{equation}

We now define the stationary response appearing in the exact correspondence.
For a positive \(a\in C^\infty(M)\) and \(f\in\mathcal D_{S_T}\), let \(z^f\) solve
\begin{equation}
\begin{cases}
z_{ss}^f-a(x)^2\Delta_g z^f=0 & \text{in }M\times(0,S_T),\\
z^f|_{s=0}=z_s^f|_{s=0}=0 & \text{in }M,\\
z^f=f & \text{on }\partial M\times(0,S_T).
\end{cases}
\label{eq:stationary-forward}
\end{equation}
Its Dirichlet-to-Neumann map is
\begin{equation}
R_{g,a}^{S_T}f
=
\partial_{\nu_g}z^f|_{\partial M\times(0,S_T)}.
\label{eq:stationary-response}
\end{equation}
Proposition~\ref{prop:forward-well-posedness-causality}, applied in the time variable \(s\) with the spatial factor \(a(x)\), gives
\[
R_{g,a}^{S_T}
:
\mathcal D_{S_T}
\longrightarrow
\mathcal D_{S_T}.
\]
Lemma~\ref{lem:effective-time-rescaling} and Lemma~\ref{lem:intertwining} place \(\Lambda_{g,a,q}^{S_T}\), \(T_q^{S_T}\), and \((T_q^{S_T})^{-1}\) on the same boundary source space.

\begin{proposition}
\label{prop:boundary-data-correspondence}
Let \(c\in\mathcal C^T(M)\), and let \(a\) and \(b\) denote its normalized factors from \eqref{eq:main-separable-speed}.
Then \(\Lambda_{g,c}^T\) and the triple
\[
\bigl(g_{\partial M},b,R_{g,a}^{S_T}\bigr)
\]
determine one another explicitly through the invertible causal operators constructed above; in particular, the passage to the stationary response loses no boundary information.
If two normalized data sets satisfy \eqref{eq:boundary-identification}, with \(P_\varphi\) defined by \eqref{eq:boundary-pushforward}, then
\begin{equation}
\varphi^*g_{2,\partial M_2}=g_{1,\partial M_1},
\qquad
b_1=b_2=:b,
\qquad
R_{g_1,a_1}^{S_T}
=
P_\varphi^{-1}R_{g_2,a_2}^{S_T}P_\varphi,
\quad
S_T=\int_0^Tb(t)\,\mathrm{d}t.
\label{eq:exact-correspondence-comparison}
\end{equation}
\end{proposition}

\begin{proof}[Proof of Proposition~\ref{prop:boundary-data-correspondence}]
Let \(q\), \(U_b^{S_T}\), and \(T_q^{S_T}\) be defined by \eqref{eq:q-p}, \eqref{eq:U-definition}, and \eqref{eq:Tq-definition}, respectively.
Fix \(f\in\mathcal D_{S_T}\), let \(z=z^f\) solve the stationary problem \eqref{eq:stationary-forward}, and set
\[
w=T_q^{S_T}z.
\]
The spatial form of Lemma~\ref{lem:intertwining}, namely \eqref{eq:full-intertwining}, gives
\[
(\partial_s^2-a(x)^2\Delta_g+q(s))w
=
T_q^{S_T}(\partial_s^2-a(x)^2\Delta_g)z
=0.
\]
The same lemma preserves the zero initial data.
Hence \(w\) satisfies the potential problem \eqref{eq:w-equation} introduced in Lemma~\ref{lem:effective-time-rescaling}, and its boundary value is
\[
w|_{\partial M\times(0,S_T)}=T_q^{S_T}f.
\]
By the definition of the potential Dirichlet-to-Neumann map \(\Lambda_{g,a,q}^{S_T}\) in Lemma~\ref{lem:effective-time-rescaling}, uniqueness gives \(\Lambda_{g,a,q}^{S_T}(T_q^{S_T}f)=\partial_{\nu_g}w\).
On the other hand, the kernel \(K(s,r)\) is independent of \(x\), and the definition \eqref{eq:stationary-response} gives
\[
\partial_{\nu_g}w
=
T_q^{S_T}(\partial_{\nu_g}z)
=
T_q^{S_T}(R_{g,a}^{S_T}f).
\]
Combining the last two identities gives, for every \(f\in\mathcal D_{S_T}\),
\begin{equation}
\Lambda_{g,a,q}^{S_T}T_q^{S_T}
=
T_q^{S_T}R_{g,a}^{S_T}.
\label{eq:second-conjugacy}
\end{equation}
Since \(T_q^{S_T}\) is invertible,
\begin{equation}
R_{g,a}^{S_T}
=
(T_q^{S_T})^{-1}\Lambda_{g,a,q}^{S_T}T_q^{S_T}.
\label{eq:stationary-from-potential}
\end{equation}
Lemma~\ref{lem:effective-time-rescaling}, specifically \eqref{eq:potential-response}, gives the first conjugacy
\[
\Lambda_{g,a,q}^{S_T}
=
U_b^{S_T}\Lambda_{g,c}^T(U_b^{S_T})^{-1}.
\]
Substituting it into \eqref{eq:stationary-from-potential} yields
\begin{equation}
R_{g,a}^{S_T}
=
(T_q^{S_T})^{-1}
U_b^{S_T}\Lambda_{g,c}^T(U_b^{S_T})^{-1}
T_q^{S_T}.
\label{eq:master-conjugacy}
\end{equation}
Multiplying \eqref{eq:master-conjugacy} on the left by \((U_b^{S_T})^{-1}T_q^{S_T}\) and on the right by \((T_q^{S_T})^{-1}U_b^{S_T}\) yields
\begin{equation}
\Lambda_{g,c}^T
=
(U_b^{S_T})^{-1}T_q^{S_T}
R_{g,a}^{S_T}
(T_q^{S_T})^{-1}U_b^{S_T}.
\label{eq:physical-from-stationary}
\end{equation}

It remains to verify the determination statements.
Lemma~\ref{lem:boundary-symbol} recovers \(g_{\partial M}\) and \(c|_{\Sigma_T}\) from \(\Lambda_{g,c}^T\), and \eqref{eq:recover-b} then recovers \(b\).
The function \(b\) determines \(S_T\), the potential \(q\) through \eqref{eq:q-b}, and the operators \(U_b^{S_T}\), \(T_q^{S_T}\), and their inverses.
Formula \eqref{eq:master-conjugacy} therefore determines \(R_{g,a}^{S_T}\).
Conversely, \eqref{eq:physical-from-stationary} reconstructs \(\Lambda_{g,c}^T\) from \(b\) and \(R_{g,a}^{S_T}\), proving the exact correspondence in both directions.

Finally, suppose that two data sets satisfy \eqref{eq:boundary-identification}.
Equations \eqref{eq:boundary-metric-speed-recovery} and \eqref{eq:common-b} give
\[
\varphi^*g_{2,\partial M_2}=g_{1,\partial M_1},
\qquad
b_1=b_2=:b.
\]
Thus the two data sets have the same \(S_T\), \(q\), and temporal operators \(U_b^{S_T}\) and \(T_q^{S_T}\).
These operators and their inverses act only in time and hence commute with \(P_\varphi\).
\enlargethispage{2\baselineskip}
Using \eqref{eq:master-conjugacy} and then \eqref{eq:boundary-identification}, we obtain
\begin{align*}
R_{g_1,a_1}^{S_T}
&=
(T_q^{S_T})^{-1}U_b^{S_T}\Lambda_{g_1,c_1}^T(U_b^{S_T})^{-1}T_q^{S_T}
\\
&=
P_\varphi^{-1}
(T_q^{S_T})^{-1}U_b^{S_T}\Lambda_{g_2,c_2}^T(U_b^{S_T})^{-1}T_q^{S_T}
P_\varphi
\\
&=
P_\varphi^{-1}R_{g_2,a_2}^{S_T}P_\varphi.
\end{align*}

The proof is complete.
\end{proof}

\subsection{Invariant response of the weighted geometry}

Writing the identified stationary response as a boundary density gives the invariant datum of the weighted geometry.
\begin{lemma}
\label{lem:weighted-stationary-response}
Let \((M,g)\) be a smooth compact connected \(n\)-dimensional Riemannian manifold with smooth nonempty boundary, let \(a\in C^\infty(M)\) be positive, and let \(S_T>0\).
Set
\[
h:=a^{-2}g,
\qquad
\varrho:=a^{n-2},
\]
and define the associated weighted Laplacian and measure by
\begin{equation}
\Delta_{h,\varrho}u
=
\varrho^{-1}\operatorname{div}_h(\varrho\nabla_hu),
\qquad
\mathrm{d}\mu_{h,\varrho}=\varrho\,\mathrm{d}V_h.
\label{eq:weighted-laplacian-definition}
\end{equation}
Then
\begin{equation}
\Delta_{h,\varrho}=a^2\Delta_g,
\qquad
\mathrm{d}\mu_{h,\varrho}=a^{-2}\mathrm{d}V_g.
\label{eq:density-wave-equation}
\end{equation}
For \(f\in\mathcal D_{S_T}\), let \(z^f\) solve
\begin{equation}
\begin{cases}
z_{ss}^f-\Delta_{h,\varrho}z^f=0
&\text{in }M\times(0,S_T),\\
z^f|_{s=0}=z_s^f|_{s=0}=0
&\text{in }M,\\
z^f=f
&\text{on }\partial M\times(0,S_T).
\end{cases}
\label{eq:weighted-wave-problem}
\end{equation}
Then \eqref{eq:weighted-wave-problem} is the initial-boundary value problem \eqref{eq:stationary-forward}, and its boundary response density is
\begin{equation}
\mathcal B_{h,\varrho}^{S_T}f
:=
\varrho\,\partial_{\nu_h}z^f\,\mathrm{d}S_h
=
(R_{g,a}^{S_T}f)\,\mathrm{d}S_g.
\label{eq:weighted-boundary-density}
\end{equation}
Consequently, \((R_{g,a}^{S_T},g_{\partial M})\) determines \(\mathcal B_{h,\varrho}^{S_T}\).
\end{lemma}

\begin{proof}
For \(h=a^{-2}g\),
\[
\nabla_hu=a^2\nabla_gu,
\qquad
\mathrm{d}V_h=a^{-n}\mathrm{d}V_g,
\qquad
\operatorname{div}_hX=a^n\operatorname{div}_g(a^{-n}X).
\]
Since \(\varrho=a^{n-2}\), these identities give
\[
\Delta_{h,\varrho}u
=
a^{2-n}a^n\operatorname{div}_g\!\left(a^{-n}a^{n-2}a^2\nabla_gu\right)
=
a^2\Delta_gu,
\qquad
\mathrm{d}\mu_{h,\varrho}
=
a^{n-2}a^{-n}\mathrm{d}V_g
=
a^{-2}\mathrm{d}V_g.
\]
Thus the differential equations in \eqref{eq:weighted-wave-problem} and \eqref{eq:stationary-forward} agree; their initial and boundary conditions are identical.
For the conformal metric \(h=a^{-2}g\), one has
\begin{equation}
\nu_h=a\nu_g,
\qquad
\mathrm{d}S_h=a^{-(n-1)}\mathrm{d}S_g,
\qquad
\varrho\partial_{\nu_h}z\,\mathrm{d}S_h
=
\partial_{\nu_g}z\,\mathrm{d}S_g.
\label{eq:acoustic-flux-density}
\end{equation}
Indeed, the last identity follows from \(\varrho=a^{n-2}\) and the first two identities.
The final equality in \eqref{eq:weighted-boundary-density} now follows from \eqref{eq:acoustic-flux-density}.
Since \(g_{\partial M}\) determines \(\mathrm{d}S_g\), the final assertion follows.
\end{proof}

Recall that \(P_\varphi\) is defined in \eqref{eq:boundary-pushforward}.
For a boundary density \(\mu_{\partial,2}\) on \(\partial M_2\), its pullback is defined by
\[
\int_{\partial M_1}\psi\,\varphi^*\mu_{\partial,2}
=
\int_{\partial M_2}(\psi\circ\varphi^{-1})\,\mu_{\partial,2}
\]
for every \(\psi\in C^\infty(\partial M_1)\).
In local boundary coordinates, if \(\mu_{\partial,2}=\sigma(z)|\mathrm{d}z|\), then
\[
\varphi^*\mu_{\partial,2}
=
\sigma(\varphi(y))|\det D\varphi(y)|\,|\mathrm{d}y|.
\]

Assume that the original Dirichlet-to-Neumann maps satisfy
\[
\Lambda_{g_1,c_1}^T
=
P_\varphi^{-1}\Lambda_{g_2,c_2}^TP_\varphi.
\]
Then Proposition~\ref{prop:boundary-data-correspondence} and Lemma~\ref{lem:weighted-stationary-response} give
\[
b_1=b_2=:b,
\qquad
\varphi^*g_{2,\partial M_2}=g_{1,\partial M_1},
\qquad
S_T=\int_0^T b(t)\,\mathrm{d}t,
\]
and, for every \(f\in C_c^\infty((0,S_T)\times\partial M_1)\),
\begin{equation}
\mathcal B_{h_1,\varrho_1}^{S_T}f
=
\varphi^*\!\left(
\mathcal B_{h_2,\varrho_2}^{S_T}(P_\varphi f)
\right).
\label{eq:weighted-response-equality}
\end{equation}
Thus the original data determine the common temporal factor \(b\) and the invariant weighted response \eqref{eq:weighted-response-equality}.

\section{Finite-time reconstruction of the weighted geometry}
\label{sec:finite-time-recovery}

Under the visibility condition, the response identity \eqref{eq:weighted-response-equality} determines \((M,h,\varrho)\) up to a weighted isometry, that is, a diffeomorphism preserving both \(h\) and \(\varrho\), and extending the boundary identification.
At the midpoint \(S_T/2\), the proof passes from the response to the terminal Gramian, labeled domains of influence, truncated boundary distance profiles, \(h\), and finally \(\varrho\,\mathrm{d}V_h\) and \(\varrho\).
This reconstruction holds for \(n\geq2\); the condition \(n\geq3\) is used only in Section~\ref{sec:completion-main-theorem}.
The argument belongs to the boundary control framework developed in \cite{belishev1991boundary,belishev1992reconstruction,belishev1997boundary,katchalov2001inverse,katchalov2004equivalence}; related continuation, reconstruction, and stability results are given in \cite{kurylev2002hyperbolic,dehoop2018recovery,stefanov1998stability,stefanov2005stable,katchalov2004energy}.
All function spaces below are over \(\mathbb R\); complex data follow by complexification, with \(\langle u,v\rangle_{L^2(\mathrm{d}\mu)}=\int u\overline v\,\mathrm{d}\mu\).

\begin{proposition}
\label{prop:finite-time-spatial-recovery}
Let \(S_T>0\), let \((M_j,h_j)\), \(j=1,2\), be smooth compact connected \(n\)-dimensional Riemannian manifolds, \(n\geq2\), with smooth nonempty boundaries, and let \(\varrho_j\in C^\infty(M_j)\) be strictly positive.
For \(j=1,2\), let \(\mathcal B_{h_j,\varrho_j}^{S_T}\) be the boundary response density for \((\partial_s^2-\Delta_{h_j,\varrho_j})z=0\) with zero initial data, as in \eqref{eq:weighted-wave-problem}--\eqref{eq:weighted-boundary-density}.
Suppose that \(\varphi:\partial M_1\to\partial M_2\) is a diffeomorphism and that, for every
\(f\in C_c^\infty((0,S_T)\times\partial M_1)\),
\[
\mathcal B_{h_1,\varrho_1}^{S_T}f
=
\varphi^*\!\left(
\mathcal B_{h_2,\varrho_2}^{S_T}(P_\varphi f)
\right).
\]
Assume also that
\begin{equation}
S_T>2\max\{\operatorname{rad}(h_1),\operatorname{rad}(h_2)\}.
\label{eq:weighted-filling-condition}
\end{equation}
Then there is a diffeomorphism \(\Phi:M_1\to M_2\) such that
\begin{equation}
\Phi|_{\partial M_1}=\varphi,
\qquad
\Phi^*h_2=h_1,
\qquad
\Phi^*\varrho_2=\varrho_1.
\label{eq:weighted-isometry-conclusion}
\end{equation}
\end{proposition}

\subsection{Geometric preliminaries: nullity of distance level sets}

The geometric reconstruction uses the following measure property of distance functions.
\begin{lemma}
\label{lem:distance-level-null}
Let \((N,h)\) be a smooth compact connected Riemannian manifold with nonempty boundary, and set
\[
d_{\partial N}(x)=\operatorname{dist}_h(x,\partial N).
\]
For every \(r>0\), the level set
\[
\{x\in N:d_{\partial N}(x)=r\}
\]
has zero Riemannian volume.
For every fixed \(z\in\overline N\) and \(r>0\), the distance sphere
\[
\{x\in N:\operatorname{dist}_h(x,z)=r\}
\]
also has zero Riemannian volume.
Consequently, both types of level sets are null for every smooth positive
weighted measure on \(N\).
\end{lemma}

\begin{proof}
The function \(d_{\partial N}\) is \(1\)-Lipschitz.
At every differentiability point \(x\in N^\circ\), compactness provides a closest boundary point and a minimizing unit-speed geodesic \(\gamma\) from \(x\) to that point.
Along its initial segment,
\[
d_{\partial N}(\gamma(s))=d_{\partial N}(x)-s.
\]
Therefore
\[
\mathrm{d}d_{\partial N}(x)(\dot\gamma(0))=-1.
\]
Together with the Lipschitz bound, this gives
\[
|\nabla_h d_{\partial N}|_h=1
\quad\text{almost everywhere in }N^\circ.
\]
Suppose that \(E=\{d_{\partial N}=r\}\) has positive volume for some \(r>0\).
At almost every density point of \(E\) where \(d_{\partial N}\) is differentiable,
the first-order expansion of \(d_{\partial N}\) on the density-one set \(E\)
implies \(\mathrm{d}d_{\partial N}=0\).
This contradicts the preceding eikonal identity.
For a fixed \(z\in\overline N\), the function \(d_z(x)=\operatorname{dist}_h(x,z)\) is also \(1\)-Lipschitz.
At every differentiability point with \(d_z(x)>0\), a minimizing geodesic from \(x\) to \(z\) gives \(|\nabla_h d_z|_h=1\).
Repeating the density point argument proves that every positive level set of \(d_z\) has zero Riemannian volume.
The assertion for a smooth positive weighted measure follows because its density
with respect to Riemannian volume is bounded on the compact manifold \(N\).
\end{proof}

\subsection{Reconstruction from the finite-time weighted response}

We prove Proposition~\ref{prop:finite-time-spatial-recovery} in four steps.

\begin{proof}[Proof of Proposition~\ref{prop:finite-time-spatial-recovery}]
\medskip
\noindent
\emph{Step 1: the response density and the terminal Gramian.}
For \(j=1,2\), let \(z_j^f\) denote the solution on \((0,S_T)\times M_j\) of the weighted wave equation generated by \(\Delta_{h_j,\varrho_j}\), with zero initial data and Dirichlet source \(f\).
Thus \(\mathcal B_{h_j,\varrho_j}^{S_T}f=\varrho_j\partial_{\nu_{h_j}}z_j^f\,\mathrm{d}S_{h_j}\).
We derive from Green's formula an identity expressing terminal inner products of waves in terms of the boundary response; see
\cite{belishev1997boundary,katchalov2004equivalence}
for the Blagovestchenskii identity.
Set
\[
\mathrm{d}\mu_j=\varrho_j\,\mathrm{d}V_{h_j},
\qquad
\mathrm{d}\mu_{j,\partial}=\varrho_j\,\mathrm{d}S_{h_j}.
\]
For smooth \(u\) and \(w\), the weighted Green identity is
\begin{equation}
\int_{M_j}(\Delta_{h_j,\varrho_j}u)w\,\mathrm{d}\mu_j
=
-\int_{M_j}\langle\nabla_{h_j}u,\nabla_{h_j}w\rangle_{h_j}\,\mathrm{d}\mu_j
+\int_{\partial M_j}\partial_{\nu_{h_j}}u\,w\,\mathrm{d}\mu_{j,\partial}.
\label{eq:weighted-green-identity}
\end{equation}
In particular, the Dirichlet realization of \(-\Delta_{h_j,\varrho_j}\) is the nonnegative self-adjoint operator associated with the closed Dirichlet energy form in \(L^2(M_j,\mathrm{d}\mu_j)\).

For controls supported in \((0,S_T/2)\times\partial M_j\), define
\[
W_j^{S_T/2} f=z_j^f(S_T/2).
\]
If
\[
I_{f,k}(t,s)
=
\langle z_j^f(t),z_j^k(s)\rangle_{L^2(M_j,\mathrm{d}\mu_j)},
\]
then \eqref{eq:weighted-green-identity} gives
\begin{equation}
(\partial_t^2-\partial_s^2)I_{f,k}(t,s)
=
\int_{\partial M_j}k(s)\,\mathcal B_{h_j,\varrho_j}^{S_T}f(t)
-\int_{\partial M_j}f(t)\,\mathcal B_{h_j,\varrho_j}^{S_T}k(s).
\label{eq:weighted-blagovestchenskii-wave}
\end{equation}
The zero initial data supply homogeneous data on the coordinate axes.
Writing the right-hand side of \eqref{eq:weighted-blagovestchenskii-wave} as \(F_{f,k}(t,s)\), the d'Alembert formula at \((S_T/2,S_T/2)\) is
\begin{equation}
I_{f,k}(S_T/2,S_T/2)
=
\frac12\int_0^{S_T/2}\int_t^{S_T-t}F_{f,k}(t,s)\,\mathrm{d}s\,\mathrm{d}t.
\label{eq:weighted-dalembert-formula}
\end{equation}
Consequently, the boundary response density on \((0,S_T)\) determines the terminal Gramian
\begin{equation}
G_j(f,k)
=
\langle W_j^{S_T/2} f,W_j^{S_T/2} k\rangle_{L^2(M_j,\mathrm{d}\mu_j)}.
\label{eq:weighted-terminal-gramian}
\end{equation}
The characteristic triangle in \eqref{eq:weighted-dalembert-formula} lies in \((0,S_T)^2\).

\medskip
\noindent\emph{Step 2: finite propagation and approximate controllability in domains of influence.}
Fix a nonempty relatively open set \(\Gamma\subset\partial M_j\) and a number \(0<r<S_T/2\).
Write
\[
M_{h_j}(\Gamma,r)
=
\{x\in M_j:\operatorname{dist}_{h_j}(x,\Gamma)<r\}
\]
and
\[
\mathscr W_j(\Gamma,r)
=
\{z_j^f(S_T/2):f\in C_c^\infty((S_T/2-r,S_T/2)\times\Gamma)\}.
\]
Finite propagation for the metric \(h_j\) governing travel times gives
\begin{equation}
\mathscr W_j(\Gamma,r)
\subset
L^2(M_{h_j}(\Gamma,r),\mathrm{d}\mu_j).
\label{eq:weighted-finite-propagation}
\end{equation}

Let \(\psi\in L^2(M_{h_j}(\Gamma,r),\mathrm{d}\mu_j)\) be
orthogonal to \(\mathscr W_j(\Gamma,r)\) in this space, extend \(\psi\) by
zero to \(M_j\), and let \(w\) be the energy solution of
\begin{equation}
\begin{cases}
w_{ss}-\Delta_{h_j,\varrho_j}w=0
&\text{in }(0,S_T/2)\times M_j,\\
w=0
&\text{on }(0,S_T/2)\times\partial M_j,\\
w(S_T/2)=0,\quad w_s(S_T/2)=\psi
&\text{in }M_j.
\end{cases}
\label{eq:weighted-backward-wave}
\end{equation}
The hidden regularity estimate for the homogeneous Dirichlet wave equation
gives the normal trace
\[
\partial_{\nu_{h_j}}w
\in
L^2((0,S_T/2)\times\partial M_j,
\mathrm{d}s\,\mathrm{d}\mu_{j,\partial})
\]
and, more precisely,
\begin{equation}
\|\partial_{\nu_{h_j}}w\|_{L^2((0,S_T/2)\times\partial M_j,
\mathrm{d}s\,\mathrm{d}\mu_{j,\partial})}
\leq C
\left(
\|w(S_T/2)\|_{H_0^1(M_j)}
+\|w_s(S_T/2)\|_{L^2(M_j,\mathrm{d}\mu_j)}
\right).
\label{eq:weighted-hidden-regularity}
\end{equation}
For smooth terminal data, the standard Rellich multiplier argument with a
vector field equal to \(\nu_{h_j}\) on the boundary produces the boundary term
\(\int |\partial_{\nu_{h_j}}w|^2\,\mathrm{d}s\,\mathrm{d}\mu_{j,\partial}\);
all interior and endpoint terms are bounded by the conserved energy.
The smooth positive weight changes only lower-order terms and replaces
\(\mathrm{d}S_{h_j}\) by
\(\mathrm{d}\mu_{j,\partial}=\varrho_j\mathrm{d}S_{h_j}\).
Density then proves \eqref{eq:weighted-hidden-regularity} for energy solutions;
see \cite[Chapter~1]{lions1988controlabilite}.
Approximating \(\psi\) by smooth
terminal velocities therefore justifies both the spacetime Green identity and
the limiting normal trace in \eqref{eq:weighted-terminal-duality}.
The spacetime Green identity yields the terminal duality formula
\begin{equation}
\langle z_j^f(S_T/2),\psi\rangle_{L^2(M_j,\mathrm{d}\mu_j)}
=
\int_{S_T/2-r}^{S_T/2}\int_\Gamma
f\,\partial_{\nu_{h_j}}w\,\mathrm{d}\mu_{j,\partial}\,\mathrm{d}s.
\label{eq:weighted-terminal-duality}
\end{equation}
Orthogonality and the arbitrariness of \(f\) imply
\[
w=\partial_{\nu_{h_j}}w=0
\quad\text{on }(S_T/2-r,S_T/2)\times\Gamma
\]
in the trace sense.

The terminal condition \(w(S_T/2)=0\) permits the odd reflection
\[
\widetilde w(s,x)
=
\begin{cases}
w(s,x),&S_T/2-r<s\leq S_T/2,\\
-w(S_T-s,x),&S_T/2\leq s<S_T/2+r.
\end{cases}
\]
The time derivatives from the two adjacent intervals agree at \(s=S_T/2\), so \(\widetilde w\) is an energy solution of the same wave equation on \((S_T/2-r,S_T/2+r)\times M_j\).
Moreover,
\begin{equation}
\widetilde w=\partial_{\nu_{h_j}}\widetilde w=0
\quad\text{on }(S_T/2-r,S_T/2+r)\times\Gamma.
\label{eq:weighted-reflected-cauchy-data}
\end{equation}
The cylinder extending to both sides of \(s=S_T/2\) in \eqref{eq:weighted-reflected-cauchy-data} yields a nondegenerate boundary double cone there.

Set
\[
m=\varrho_j^{1/2},
\qquad
\zeta=m\widetilde w,
\qquad
V=m^{-1}\Delta_{h_j}m.
\]
Expanding \(\Delta_{h_j,\varrho_j}\) with \(\varrho_j=m^2\) yields
\begin{equation}
m\Delta_{h_j,\varrho_j}m^{-1}
=
\Delta_{h_j}-m^{-1}(\Delta_{h_j}m),
\label{eq:weighted-scalar-conjugation}
\end{equation}
and hence
\[
(\partial_s^2-\Delta_{h_j}+V)\zeta=0.
\]
Since \(\widetilde w=0\) on the boundary, \eqref{eq:weighted-reflected-cauchy-data} is equivalent to
\[
\zeta=\partial_{\nu_{h_j}}\zeta=0
\quad\text{on }(S_T/2-r,S_T/2+r)\times\Gamma.
\]
The formulation in terms of boundary domains needed here is stated explicitly in
\cite[Theorem~4]{katchalov2004energy}; see also
\cite[Theorem~3.16]{katchalov2001inverse}.
It is obtained from Tataru's local
unique continuation theorem across noncharacteristic hypersurfaces
\cite{tataru1995unique} and applies to time-independent wave operators with
smooth lower-order terms.
To obtain the full spacetime cone, fix \(s_*\) with \(|s_*-S_T/2|<r\) and set
\[
\tau_*=r-|s_*-S_T/2|.
\]
The interval \((s_*-\tau_*,s_*+\tau_*)\) is contained in
\((S_T/2-r,S_T/2+r)\).
Applying the cited boundary theorem after translating
time by \(s_*\) gives vanishing at time \(s_*\) whenever
\(\operatorname{dist}_{h_j}(x,\Gamma)<\tau_*\).
Since \(s_*\)
is arbitrary, this gives the double cone conclusion
\begin{equation}
\zeta(s,x)=0
\quad\text{if}\quad
|s-S_T/2|+\operatorname{dist}_{h_j}(x,\Gamma)<r
\label{eq:weighted-double-cone}
\end{equation}
for smooth solutions of the operator \(\partial_s^2-\Delta_{h_j}+V\).
To extend the double cone conclusion from smooth to energy solutions, fix \(\varepsilon>0\) and let \(\zeta_\delta=\chi_\delta*_s\zeta\), where \(\chi_\delta\) is an even time mollifier of scale \(0<\delta<\varepsilon\).
We work on the shortened cylinder \((S_T/2-r+\varepsilon,S_T/2+r-\varepsilon)\times M_j\), where the convolution is well defined.
Because the coefficients are time independent, \(\zeta_\delta\) satisfies the same equation.
Time convolution also commutes with the boundary traces, so
\[
\zeta_\delta=\partial_{\nu_{h_j}}\zeta_\delta=0
\quad\text{on }(S_T/2-r+\varepsilon,S_T/2+r-\varepsilon)\times\Gamma.
\]
The function \(\zeta_\delta\) is smooth in time with values in the energy domain.
The identity \(\Delta_{h_j}\zeta_\delta=\partial_s^2\zeta_\delta+V\zeta_\delta\), together with the homogeneous Dirichlet condition, allows Dirichlet elliptic regularity to be iterated and shows that \(\zeta_\delta\) is smooth in the spatial variables as well.
The smooth result therefore applies where \(|s-S_T/2|+\operatorname{dist}_{h_j}(x,\Gamma)<r-\varepsilon\).
The mollified solutions converge to \(\zeta\) in \(C(H^1)\cap C^1(L^2)\) on compact time intervals, so letting first \(\delta\downarrow0\) and then \(\varepsilon\downarrow0\) proves \eqref{eq:weighted-double-cone} for \(\zeta\).
This is also the unique continuation step used to prove approximate
controllability in domains of influence in
\cite[Theorem~5]{katchalov2004energy}.
Taking a time neighborhood of \(s=S_T/2\) inside \eqref{eq:weighted-double-cone} gives
\[
0=\zeta_s(S_T/2,x)=m(x)\psi(x)
\quad\text{for almost every }x\in M_{h_j}(\Gamma,r).
\]
Because \(m>0\), this proves \(\psi=0\) almost everywhere on \(M_{h_j}(\Gamma,r)\).
Together with \eqref{eq:weighted-finite-propagation}, the duality argument proves
\begin{equation}
\overline{\mathscr W_j(\Gamma,r)}^{\,L^2(M_j,\mathrm{d}\mu_j)}
=
L^2(M_{h_j}(\Gamma,r),\mathrm{d}\mu_j).
\label{eq:weighted-restricted-controllability}
\end{equation}
Equation~\eqref{eq:weighted-restricted-controllability} is the required approximate controllability conclusion.

\medskip
\noindent\emph{Step 3: a common Hilbert realization and reconstruction of the metric from travel times.}
Use \(\varphi\) to identify the two control spaces, and set
\[
\mathscr C=C_c^\infty((0,S_T/2)\times\partial M_1),
\qquad
W_1^{S_T/2} f=z_1^f(S_T/2),
\qquad
W_2^{S_T/2} f=z_2^{P_\varphi f}(S_T/2).
\]
By Step 1, the equality of the boundary response densities implies
\begin{equation}
G(f,k)
=
\langle W_1^{S_T/2} f,W_1^{S_T/2} k\rangle_{L^2(M_1,\mathrm{d}\mu_1)}
=
\langle W_2^{S_T/2} f,W_2^{S_T/2} k\rangle_{L^2(M_2,\mathrm{d}\mu_2)}.
\label{eq:common-weighted-gramian}
\end{equation}
Define
\[
\mathscr N
=
\{f\in\mathscr C:G(f,f)=0\},
\qquad
\mathscr H_G
=
\overline{\mathscr C/\mathscr N}^{\,G}.
\]
Let \(Q:\mathscr C\to\mathscr C/\mathscr N\) be the quotient map.
The maps
\[
\mathcal V_j[f]=W_j^{S_T/2} f
\]
are isometries from \(\mathscr C/\mathscr N\) into \(H_j=L^2(M_j,\mathrm{d}\mu_j)\).
Choose \(r_0\) with
\[
\max\{\operatorname{rad}(h_1),\operatorname{rad}(h_2)\}<r_0<S_T/2.
\]
Equation \eqref{eq:weighted-restricted-controllability}, applied with \(\Gamma=\partial M_j\) and \(r=r_0\), shows that both terminal control maps have dense range.
Consequently, each \(\mathcal V_j\) extends to a unitary map \(\mathscr H_G\to H_j\), and
\begin{equation}
U
=
\mathcal V_2\mathcal V_1^{-1}:H_1\longrightarrow H_2
\label{eq:common-unitary}
\end{equation}
is the unitary canonically induced by the common Gramian and satisfies
\(UW_1^{S_T/2} f=W_2^{S_T/2} f\).

For open \(\Gamma\subset\partial M_1\) and \(0<r<S_T/2\), set
\[
\mathscr C(\Gamma,r)
=
\{f\in\mathscr C:\operatorname{supp}f\subset(S_T/2-r,S_T/2)\times\Gamma\}
\]
and
\[
\mathscr H_G(\Gamma,r)
=
\overline{Q(\mathscr C(\Gamma,r))}^{\,\mathscr H_G}.
\]
Equation \eqref{eq:weighted-restricted-controllability} gives
\begin{align}
\mathcal V_1\mathscr H_G(\Gamma,r)
&=L^2(M_{h_1}(\Gamma,r),\mathrm{d}\mu_1),\notag\\
\mathcal V_2\mathscr H_G(\Gamma,r)
&=L^2(M_{h_2}(\varphi(\Gamma),r),\mathrm{d}\mu_2).
\label{eq:weighted-labelled-supports}
\end{align}
Thus the boundary response density determines the labeled family of influence
subspaces in one common Hilbert space.

We next recover the common representation by truncated boundary distances from the
labeled subspaces in \eqref{eq:weighted-labelled-supports}.
Put
\[
R=\frac{S_T}{2},
\qquad
\varphi_1=\operatorname{Id}_{\partial M_1},
\qquad
\varphi_2=\varphi,
\]
and, for \(z\in\partial M_1\) and \(0<s<R\), define
\[
\mathbb B_j(z,s)
=
\{x\in M_j:\operatorname{dist}_{h_j}(x,\varphi_j(z))<s\}.
\]
For every measurable set \(A\subset M_j\), let \(\Pi_A\) denote the
orthogonal projection of \(L^2(M_j,\mathrm{d}\mu_j)\) onto
\(L^2(A,\mathrm{d}\mu_j)\); equivalently, \(\Pi_A\) is
multiplication by the characteristic function of \(A\).
These balls centered at boundary points are determined by the labeled domains of
influence.
Indeed, choose neighborhoods \(\Gamma_m\) of \(z\) that shrink to
\(z\), and numbers \(\epsilon_m\downarrow0\), so that
\[
\operatorname{diam}_{h_1}(\Gamma_m)<\frac{\epsilon_m}{2},
\qquad
\operatorname{diam}_{h_2}(\varphi(\Gamma_m))<\frac{\epsilon_m}{2}.
\]
The triangle inequality gives, for \(j=1,2\),
\begin{equation}
\mathbb B_j(z,s)
=
\bigcup_{m:\,\epsilon_m<s}
M_{h_j}\bigl(\varphi_j(\Gamma_m),s-\epsilon_m\bigr).
\label{eq:boundary-ball-from-domains}
\end{equation}
For the nontrivial inclusion, if the distance from \(x\) to
\(\varphi_j(z)\) is strictly smaller than \(s\), then for all sufficiently
large \(m\) it is smaller than \(s-\epsilon_m\), and
\(\varphi_j(z)\in\varphi_j(\Gamma_m)\).
Conversely, membership in a set on
the right-hand side and the diameter bound imply
\(\operatorname{dist}_{h_j}(x,\varphi_j(z))<s-\epsilon_m/2<s\).
Taking closed spans of the corresponding \(L^2\)-spaces in
\eqref{eq:boundary-ball-from-domains} and using
\eqref{eq:weighted-labelled-supports} yields
\begin{equation}
U \Pi_{\mathbb B_1(z,s)}U^{-1}=\Pi_{\mathbb B_2(z,s)},
\qquad z\in\partial M_1,\quad 0<s<R.
\label{eq:boundary-ball-projections}
\end{equation}

Next define the truncated boundary distance maps
\begin{equation}
\mathcal R_j^R:M_j^R\longrightarrow C(\partial M_1),
\qquad
\mathcal R_j^R(x)(z)
=
\min\bigl\{\operatorname{dist}_{h_j}(x,\varphi_j(z)),R\bigr\},
\label{eq:truncated-boundary-distance-map}
\end{equation}
where
\[
M_j^R=\{x\in M_j:\operatorname{dist}_{h_j}(x,\partial M_j)<R\}.
\]
We claim that
\begin{equation}
\mathcal R_1^R(M_1^R)=\mathcal R_2^R(M_2^R).
\label{eq:equal-truncated-distance-images}
\end{equation}
To prove the claim, first observe that
\eqref{eq:boundary-ball-projections} and the identities for products and
complements of multiplication projections determine whether any finite set of
the form
\begin{equation}
\bigcap_{\ell=1}^N
\left\{
x\in M_j:
a_\ell<
\operatorname{dist}_{h_j}(x,\varphi_j(z_\ell))
<b_\ell
\right\},
\qquad 0\leq a_\ell<b_\ell<R,
\label{eq:finite-distance-cell}
\end{equation}
is empty.
More precisely, its multiplication projection is obtained by
multiplying the projections in \eqref{eq:boundary-ball-projections} and their
orthogonal complements.
Lemma~\ref{lem:distance-level-null} shows that every
distance sphere with positive radius has zero \(n\)-dimensional Riemannian
measure.
Thus open or closed choices at the endpoints do not change this
projection.
Since
\eqref{eq:finite-distance-cell} is open, it is nonempty if and only if the
resulting projection is nonzero.
Hence the answer is the same for \(j=1\)
and \(j=2\).

Fix \(x\in M_1^R\), set \(r=\mathcal R_1^R(x)\), and choose
\(z_0\in\partial M_1\) with
\[
\operatorname{dist}_{h_1}(x,z_0)
=
\operatorname{dist}_{h_1}(x,\partial M_1)<R.
\]
Let \(\{z_\ell\}_{\ell=1}^\infty\) be dense in \(\partial M_1\), with
\(z_1=z_0\).
For each \(m\), choose \(0<\eta_m<1/m\) so small that, for
\(1\leq\ell\leq m\), the interval
\((r(z_\ell)-\eta_m,r(z_\ell)+\eta_m)\) lies in \((0,R)\) whenever
\(0<r(z_\ell)<R\).
On \(M_j\), impose the following condition for each
\(\ell=1,\ldots,m\):
\[
\begin{cases}
\operatorname{dist}_{h_j}(\,\cdot\,,\varphi_j(z_\ell))<\eta_m,
& r(z_\ell)=0,\\
\bigl|\operatorname{dist}_{h_j}(\,\cdot\,,\varphi_j(z_\ell))-r(z_\ell)\bigr|<\eta_m,
&0<r(z_\ell)<R,\\
\operatorname{dist}_{h_j}(\,\cdot\,,\varphi_j(z_\ell))>R-\eta_m,
&r(z_\ell)=R.
\end{cases}
\]
The intersection of these finitely many open conditions contains \(x\) when
\(j=1\), including when \(x\in\partial M_1\) and hence
\(r(z_1)=0\).
Its multiplication projection is obtained from the ball
projections in \eqref{eq:boundary-ball-projections} and their complements;
as above, the boundary spheres are null sets.
The corresponding projection
for \(j=2\) is therefore nonzero, so the intersection is nonempty.
Choose a
point \(x_m'\) in it.
Then
\begin{equation}
\left|
\mathcal R_2^R(x_m')(z_\ell)-r(z_\ell)
\right|<\frac1m,
\qquad \ell=1,\ldots,m.
\label{eq:finite-profile-approximation}
\end{equation}
Thus the cases \(r(z_\ell)=0\) and \(r(z_\ell)=R\) are both covered by the
same projection test.
Compactness of \(M_2\) gives a convergent subsequence,
say \(x_m'\to x'\).
Passing to the limit in
\eqref{eq:finite-profile-approximation} gives equality at every \(z_\ell\),
and the \(1\)-Lipschitz dependence of the distance functions on the boundary
variable extends the equality to all \(z\in\partial M_1\).
Moreover,
\[
\operatorname{dist}_{h_2}(x',\varphi(z_0))
=r(z_0)<R,
\]
so \(x'\in M_2^R\).
Thus
\(\mathcal R_1^R(M_1^R)\subset\mathcal R_2^R(M_2^R)\); interchanging the two
manifolds proves \eqref{eq:equal-truncated-distance-images}.

We now invoke the reconstruction from truncated boundary distance data in
Katchalov--Kurylev--Lassas in its precise form.
In the notation of
\cite[Theorem~6, the paragraph following formula~(74), and Lemma~6]{katchalov2004energy},
with their truncation parameter \(T\) replaced here by \(R\), the hyperbolic
boundary form determines the set of truncated boundary distance functions;
the truncated boundary distance map is a homeomorphism onto this set, suitable
evaluation functions
\[
E_z(r)=r(z),
\qquad r\in\mathcal R_j^R(M_j^R),
\]
give local smooth coordinates, and the covectors of unit length supplied by these
distance functions determine the metric.
Lemma~6 of that paper concludes that
the Riemannian manifold \((M_j^R,h_j)\) can be reconstructed from this data.
The detailed construction for compact manifolds is given in
\cite[Sections~3.8.2--3.8.3 and~4.2]{katchalov2001inverse}.

Since \eqref{eq:equal-truncated-distance-images} gives the same profile set for
\(j=1,2\), the cited reconstruction equips it with a common smooth structure
for which both \(\mathcal R_j^R\) are diffeomorphisms, including up to the
boundary.
Its metric is read from the common image as follows.
In a coordinate system \(q^a=E_{z_a}\), let
\[
H_j=(h_j^{ab})_{a,b=1}^n
\]
be the inverse metric.
Whenever \(E_z<R\) and the associated distance
function is smooth, the eikonal equation gives
\begin{equation}
\sum_{a,b=1}^n
h_j^{ab}
\frac{\partial E_z}{\partial q^a}
\frac{\partial E_z}{\partial q^b}
=1.
\label{eq:eikonal-metric-reconstruction}
\end{equation}
By the cited reconstruction result based on boundary distances, the available covectors
\(\mathrm{d}E_z\) contain an open subset of the unit cotangent sphere and
hence determine the inverse metric through
\eqref{eq:eikonal-metric-reconstruction}.
Since the evaluation functions are identical on the common image
in \eqref{eq:equal-truncated-distance-images}, the two systems have the same
coefficients and therefore recover the same metric.

We may consequently define
\begin{equation}
\Phi^R
=
(\mathcal R_2^R)^{-1}\circ\mathcal R_1^R:
(M_1^R,h_1)\longrightarrow(M_2^R,h_2).
\label{eq:finite-depth-isometry-from-profiles}
\end{equation}
The preceding coordinate argument shows that \(\Phi^R\) is a smooth isometry.
If \(z\in\partial M_1\), then
\(\mathcal R_1^R(z)(z)=0\); equality of the profiles forces
\(\Phi^R(z)=\varphi(z)\).
Thus \(\Phi^R|_{\partial M_1}=\varphi\).
Finally, equality of the truncated profiles implies, whenever \(r<R\),
\begin{equation}
\Phi^R\bigl(M_{h_1}(\Gamma,r)\bigr)
=
M_{h_2}(\varphi(\Gamma),r)
\label{eq:finite-depth-domain-map}
\end{equation}
for every open \(\Gamma\subset\partial M_1\).

Condition \eqref{eq:weighted-filling-condition} gives \(M_j^R=M_j\).
We may
therefore write \(\Phi=\Phi^R\) and obtain
\begin{equation}
\Phi:M_1\longrightarrow M_2,
\qquad
\Phi|_{\partial M_1}=\varphi,
\qquad
\Phi^*h_2=h_1.
\label{eq:recovered-weighted-isometry}
\end{equation}
In particular, by \eqref{eq:finite-depth-domain-map},
\begin{equation}
\Phi(M_{h_1}(\Gamma,r))
=
M_{h_2}(\varphi(\Gamma),r)
\label{eq:weighted-domain-map}
\end{equation}
for every open \(\Gamma\subset\partial M_1\) and \(0<r<S_T/2\).

\medskip
\noindent\emph{Step 4: recovery of the weighted measure and the weight.}
Let \(1_j\) be the constant function one in \(H_j\).
For \(f\in\mathscr C\), set \(f_1=f\) and \(f_2=P_\varphi f\), and define
\[
J_{j,f}(t)
=
\langle z_j^{f_j}(t),1_j\rangle_{H_j}.
\]
Since \(\Delta_{h_j,\varrho_j}1_j=0\), Green's identity gives
\[
J_{j,f}''(t)
=
\int_{\partial M_j}\mathcal B_{h_j,\varrho_j}^{S_T}f_j(t).
\]
The zero initial conditions imply
\begin{equation}
\langle z_j^{f_j}(S_T/2),1_j\rangle_{H_j}
=
\int_0^{S_T/2}(S_T/2-t)
\int_{\partial M_j}\mathcal B_{h_j,\varrho_j}^{S_T}f_j(t)\,\mathrm{d}t.
\label{eq:weighted-constant-pairing}
\end{equation}
Equality of the response densities and the pullback convention therefore imply
\[
\langle W_1^{S_T/2} f,1_1\rangle_{H_1}
=
\langle W_2^{S_T/2} f,1_2\rangle_{H_2}.
\]
Using \eqref{eq:common-unitary} and the density of the terminal waves gives
\begin{equation}
U1_1=1_2.
\label{eq:weighted-constant-unitary}
\end{equation}

Recall the multiplication projections \(\Pi_A\) introduced in Step~3.
Equations \eqref{eq:weighted-labelled-supports}--\eqref{eq:weighted-domain-map} imply
\begin{equation}
U\Pi_{M_{h_1}(\Gamma,r)}U^{-1}
=
\Pi_{M_{h_2}(\varphi(\Gamma),r)}.
\label{eq:weighted-generator-projections}
\end{equation}
Let \(\mathfrak A_1\) be the algebra generated by all sets \(M_{h_1}(\Gamma,r)\) with \(\Gamma\subset\partial M_1\) open and \(0<r<S_T/2\).
For measurable sets \(A,B\subset M_1\), the corresponding multiplication projections satisfy
\[
\Pi_{A\cap B}=\Pi_A\Pi_B,
\qquad
\Pi_{M_1\setminus A}=I-\Pi_A,
\qquad
\Pi_{A\cup B}=\Pi_A+\Pi_B-\Pi_A\Pi_B.
\]
The same identities hold on \(M_2\), and \(\Phi\) respects intersections, complements, and finite unions.
Starting from \eqref{eq:weighted-generator-projections}, induction over finite Boolean operations therefore gives
\begin{equation}
U\Pi_AU^{-1}
=
\Pi_{\Phi(A)}
\quad\text{for every }A\in\mathfrak A_1.
\label{eq:weighted-algebra-projections}
\end{equation}
It follows from \eqref{eq:weighted-constant-unitary} that
\begin{equation}
\mu_2(\Phi(A))
=
\|\Pi_{\Phi(A)}1_2\|_{H_2}^2
=
\|\Pi_A1_1\|_{H_1}^2
=
\mu_1(A)
\label{eq:weighted-measure-algebra}
\end{equation}
for every \(A\in\mathfrak A_1\).

It remains to prove \(\sigma(\mathfrak A_1)=\mathcal B(M_1)\).
Choose a countable dense set \(D\subset\partial M_1\).
For every \(z\in D\), choose open boundary neighborhoods such that
\[
z\in\Gamma_m(z),
\qquad
\overline{\Gamma_{m+1}(z)}\subset\Gamma_m(z),
\qquad
\operatorname{diam}_{h_1}\Gamma_m(z)\longrightarrow0.
\]
For every rational \(0<q<S_T/2\), one has
\begin{equation}
\{x:\operatorname{dist}_{h_1}(x,z)<q\}
=
\bigcup_{\substack{p\in\mathbb Q\\0<p<q}}
\bigcap_{m=1}^\infty M_{h_1}(\Gamma_m(z),p).
\label{eq:weighted-borel-generation}
\end{equation}
Indeed, if \(\operatorname{dist}_{h_1}(x,z)<q\), choose a rational number \(p\) such that \(\operatorname{dist}_{h_1}(x,z)<p<q\); since \(z\in\Gamma_m(z)\), the point \(x\) belongs to \(M_{h_1}(\Gamma_m(z),p)\) for every \(m\).
Conversely, if \(x\) belongs to the right-hand side for some \(p<q\), choose \(z_m\in\Gamma_m(z)\) with \(\operatorname{dist}_{h_1}(x,z_m)<p\); then \(z_m\to z\), and hence \(\operatorname{dist}_{h_1}(x,z)\leq p<q\).
Moreover, the map
\[
\mathcal R_D:M_1\longrightarrow[0,S_T/2]^D,
\qquad
\mathcal R_D(x)(z)
=
\min\{\operatorname{dist}_{h_1}(x,z),S_T/2\},
\]
is injective.
Indeed, suppose that \(\mathcal R_D(x)=\mathcal R_D(y)\).
Equality on \(D\) extends to the whole boundary by continuity in \(z\).
If \(z_0\) minimizes the boundary distance from \(x\), then \(\operatorname{dist}_{h_1}(x,z_0)<S_T/2\), and equality of the truncated distance functions makes \(z_0\) a closest boundary point for \(y\) at the same distance.
If this distance is zero, then \(x=y=z_0\).
If it is positive, the first variation formula shows that the minimizing geodesics from \(z_0\) to \(x\) and \(y\) both start in the unique inward unit normal direction.
The two points lie at the same parameter value on this normal geodesic and therefore coincide.
Since \(M_1\) is compact and \([0,S_T/2]^D\) is Hausdorff, \(\mathcal R_D\) is a topological embedding.
For \(z\in D\) and rational \(0<q<S_T/2\), equation \eqref{eq:weighted-borel-generation} gives
\[
\mathcal R_D^{-1}\!\left(
\{r\in[0,S_T/2]^D:r(z)<q\}
\right)
=
\{x\in M_1:\operatorname{dist}_{h_1}(x,z)<q\}
\in
\sigma(\mathfrak A_1).
\]
Because \(D\) is countable, these rational coordinate cylinders generate the Borel sigma-algebra of \([0,S_T/2]^D\).
The embedding property therefore gives \(\mathcal B(M_1)\subset\sigma(\mathfrak A_1)\).
The reverse inclusion holds because every domain of influence \(M_{h_1}(\Gamma,r)\) is open.
Consequently,
\begin{equation}
\sigma(\mathfrak A_1)
=
\mathcal B(M_1).
\label{eq:weighted-algebra-generates-borel}
\end{equation}

Define a finite Borel measure on \(M_1\) by
\[
\nu(E)=\mu_2(\Phi(E)),
\qquad
E\in\mathcal B(M_1).
\]
Equation \eqref{eq:weighted-measure-algebra} says that \(\nu\) and \(\mu_1\) agree on the algebra \(\mathfrak A_1\).
An algebra is a \(\pi\)-system, so the \(\pi\)--\(\lambda\) theorem and \eqref{eq:weighted-algebra-generates-borel} imply that \(\nu=\mu_1\) on every Borel set.
Equivalently,
\begin{equation}
\mu_1=\Phi^*\mu_2.
\label{eq:weighted-measure-pullback}
\end{equation}
Using \(\Phi^*h_2=h_1\), we have \(\Phi^*\mathrm{d}V_{h_2}=\mathrm{d}V_{h_1}\), and hence
\[
\varrho_1\,\mathrm{d}V_{h_1}
=
\mu_1
=
\Phi^*\mu_2
=
(\varrho_2\circ\Phi)\,\mathrm{d}V_{h_1}.
\]
Thus \(\varrho_1=\varrho_2\circ\Phi\) almost everywhere, and smoothness gives equality on all of \(M_1\).
This proves \eqref{eq:weighted-isometry-conclusion}.

The proof is complete.
\end{proof}

Thus \((M,h,\varrho)\) is reconstructed for every \(n\geq2\).
Section~\ref{sec:completion-main-theorem} uses \(n\geq3\) to recover \(a\), \(g\), \(c\), and \(\mathbf G_{g,c}\).

\section{Recovery of the original Lorentzian metric}
\label{sec:completion-main-theorem}

The recovered weighted geometry \((h,\varrho)\) still encodes the unknown pair \((g,a)\).
The relation \(\varrho=a^{n-2}\) separates these coefficients precisely when \(n\geq3\), and the recovered temporal factor \(b\) then determines the full time-dependent speed:
\[
a=\varrho^{1/(n-2)},
\qquad
g=a^2h,
\qquad
c=ab.
\]
We now combine this structural inversion with the temporal and spatial recovery established above to prove Theorem~\ref{thm:simultaneous-recovery}.

\begin{proof}[Proof of Theorem~\ref{thm:simultaneous-recovery}]
Let \(n\geq3\), let \(c_j=a_jb_j\in\mathcal C^T(M_j)\), \(j=1,2\), and let \(\varphi:\partial M_1\to\partial M_2\) be a boundary diffeomorphism.
Assume the data identity \eqref{eq:boundary-identification} and the visibility condition \eqref{eq:main-finite-visibility}.
Set
\[
h_j=a_j^{-2}g_j,
\qquad
\varrho_j=a_j^{n-2}.
\]
Applying Proposition~\ref{prop:boundary-data-correspondence} to \eqref{eq:boundary-identification}, and then Lemma~\ref{lem:weighted-stationary-response}, gives
\[
b_1=b_2=:b,
\qquad
S_T=\int_0^T b(t)\,\mathrm{d}t,
\]
together with the identity \eqref{eq:weighted-response-equality} for the boundary response densities.
Because \(b_1=b_2=b\), the two accumulated effective times in \eqref{eq:main-finite-visibility} are both equal to \(S_T\).
Hence the visibility hypothesis gives
\[
S_T>2\max\{\operatorname{rad}(h_1),\operatorname{rad}(h_2)\}.
\]
Proposition~\ref{prop:finite-time-spatial-recovery} now applies and yields a diffeomorphism \(\Phi:M_1\to M_2\) satisfying
\[
\Phi|_{\partial M_1}=\varphi,
\qquad
\Phi^*h_2=h_1,
\qquad
\Phi^*\varrho_2=\varrho_1.
\]
Since \(n\geq3\), positivity and the recovered weight give
\[
\Phi^*a_2
=
(\Phi^*\varrho_2)^{1/(n-2)}
=
\varrho_1^{1/(n-2)}
=
a_1.
\]
Using \(g_j=a_j^2h_j\), we then obtain
\[
\Phi^*g_2
=
(\Phi^*a_2)^2\Phi^*h_2
=
a_1^2h_1
=
g_1.
\]
For every \((x,t)\in M_1\times[0,T]\), the identities \(b_1=b_2\) and \(\Phi^*a_2=a_1\) give
\[
c_2(\Phi(x),t)
=
a_2(\Phi(x))b_2(t)
=
a_1(x)b_1(t)
=
c_1(x,t).
\]
For \(\widehat\Phi=\Phi\times\operatorname{Id}_{(0,T)}\), these identities give
\[
\widehat\Phi^*\mathbf G_2
=
-c_2(\Phi(x),t)^2\mathrm{d}t^2+\Phi^*g_2
=
-c_1(x,t)^2\mathrm{d}t^2+g_1
=
\mathbf G_1.
\]
This proves Theorem~\ref{thm:simultaneous-recovery}.
\end{proof}

\section{Optimality and scope of the main theorem}
\label{sec:consequences-sharpness}

This section proves Corollary~\ref{cor:euclidean-recovery-consequences} and Theorem~\ref{thm:sharpness-obstructions} through precise recovery and obstruction results within the admissible multiplicatively separable class.
The arguments separate the two-dimensional conformal gauge, recovery on a fixed Euclidean background, the uniform Euclidean observation criterion, and the endpoint ambiguity on unknown manifolds.

\subsection{The dimensional restriction: a conformal gauge in dimension two}

The dimensional restriction enters only when the original coefficients \((g,a)\) are separated from the recovered pair \((h,\varrho)\).
Since \(\varrho=a^{n-2}\), it determines \(a\) for \(n\geq3\) but becomes identically one in dimension two.
The analogous two-dimensional conformal ambiguity is classical for elliptic Dirichlet-to-Neumann data \cite{lassas2001determining}.
In the present wave setting, the conformal gauge
\[
g_2=\kappa^2g_1,
\qquad
c_2=\kappa c_1,
\qquad
\mathbf G_2=\kappa^2\mathbf G_1,
\]
realizes the same loss directly at the level of the original finite-time Dirichlet-to-Neumann map when \(\kappa=1\) on \(\partial M\).

\begin{theorem}
\label{thm:two-dimensional-conformal-obstruction}
Let \((M,g_1)\) be a smooth compact connected two-dimensional Riemannian manifold with smooth nonempty boundary.
Let \(a_1\in C^\infty(M)\) and \(b\in C^\infty([0,\infty))\) be strictly positive with \(b(0)=1\), and set \(c_1(x,t)=a_1(x)b(t)\).
For every positive \(\kappa\in C^\infty(M)\) satisfying \(\kappa|_{\partial M}=1\), define
\[
g_2=\kappa^2g_1,
\qquad
a_2=\kappa a_1,
\qquad
c_2=a_2b.
\]
Then the full boundary Dirichlet-to-Neumann maps satisfy
\[
\Lambda_{g_1,c_1}^T
=
\Lambda_{g_2,c_2}^T
\]
for every finite \(T>0\).
If, in addition, \(\kappa\geq1\) on \(M\) and \(\kappa>1\) on a nonempty open subset of the interior, then the two pairs \((g_1,c_1)\) and \((g_2,c_2)\) are not related by any boundary-fixing diffeomorphism.
\end{theorem}

\begin{proof}[Proof of Theorem~\ref{thm:two-dimensional-conformal-obstruction}]
Let \((M,g_1)\), \(c_1=a_1b\), and \(\kappa\) satisfy the hypotheses of the theorem, and define \(g_2\), \(a_2\), and \(c_2\) as in its statement.
Since \(M\) is two-dimensional, \(\Delta_{g_2}=\kappa^{-2}\Delta_{g_1}\); since \(c_2=\kappa c_1\), it follows that
\[
c_2^{-2}\partial_t^2-\Delta_{g_2}
=
\kappa^{-2}\bigl(c_1^{-2}\partial_t^2-\Delta_{g_1}\bigr).
\]
Fix a finite \(T>0\) and a boundary source \(f\).
By uniqueness, the two initial-boundary value problems have the same solution \(u^f\).
Moreover, \(\nu_{g_2}=\kappa^{-1}\nu_{g_1}\), and \(\kappa=1\) on \(\partial M\); hence
\[
\Lambda_{g_2,c_2}^Tf
=
\partial_{\nu_{g_2}}u^f
=
\partial_{\nu_{g_1}}u^f
=
\Lambda_{g_1,c_1}^Tf.
\]
Since \(f\) and \(T\) were arbitrary, the finite-time maps agree for every \(T>0\).

For the last claim, keep the manifold and the first model fixed, and choose \(\kappa\geq1\) so that \(\kappa=1\) on \(\partial M\) and \(\kappa>1\) on a nonempty interior open set.
Then
\[
\operatorname{Vol}_{g_2}(M)
=
\int_M\kappa^2\,\mathrm{d}V_{g_1}
>
\operatorname{Vol}_{g_1}(M).
\]
Any boundary-fixing diffeomorphism relating the two pairs would satisfy \(\Phi^*g_2=g_1\) and therefore force the two volumes to be equal, a contradiction.
Thus this choice gives a pair that is not related by any boundary-fixing diffeomorphism, proving the final assertion.
\end{proof}

\subsection{Removal of the geometric gauge on a fixed Euclidean background}
\label{subsec:fixed-euclidean-recovery}

The obstruction above varies \(g\) and \(a\) in concert.
On a prescribed Euclidean domain, neither this conformal ambiguity nor a nontrivial boundary-fixing change of coordinates is available; recovery therefore extends to \(n\geq2\), and equality in the visibility condition is sufficient because the Euclidean background is fixed.
Fix a smooth bounded connected domain \(\Omega\subset\mathbb R^n\), and let \(c_j\in\mathcal C^T(\overline\Omega)\) have normalized factors \(a_j,b_j\), \(j=1,2\).
Set
\[
h_a:=a^{-2}\mathrm{d}x^2,
\qquad
\operatorname{rad}(h_a)
=
\max_{x\in\overline\Omega}\operatorname{dist}_{h_a}(x,\partial\Omega),
\]
and let
\[
r_\Omega
=
\operatorname{rad}(\mathrm{d}x^2)
=
\max_{x\in\overline\Omega}
\operatorname{dist}_{\mathrm{d}x^2}(x,\partial\Omega)
\]
be the Euclidean inradius.
We abbreviate
\[
\Lambda_c^T=\Lambda_{\mathrm{d}x^2,c}^T.
\]
\begin{corollary}
\label{cor:fixed-euclidean-recovery}
Let \(T>0\) and \(n\geq2\), let \(\Omega\subset\mathbb R^n\) be a smooth bounded connected domain, and for \(j=1,2\) let
\[
c_j(x,t)=a_j(x)b_j(t)
\]
on \(\overline\Omega\times[0,T]\), where \(a_j\in C^\infty(\overline\Omega)\) and \(b_j\in C^\infty([0,T])\) are strictly positive and \(b_j(0)=1\).
Set \(h_j=a_j^{-2}\mathrm{d}x^2\).
Assume that
\begin{equation}
\int_0^T b_j(t)\,\mathrm{d}t
\geq
2\operatorname{rad}(h_j),
\qquad
j=1,2,
\label{eq:fixed-euclidean-visibility}
\end{equation}
and that the full boundary Dirichlet-to-Neumann maps agree:
\[
\Lambda_{c_1}^T=\Lambda_{c_2}^T.
\]
Then the two wave speeds agree throughout the observation cylinder:
\[
c_1=c_2
\quad\text{on }\overline\Omega\times[0,T].
\]
\end{corollary}

For \(\tau>0\), write
\[
R_a^\tau:=R_{\mathrm{d}x^2,a}^\tau.
\]
For a relatively open set \(\Gamma\subset\partial\Omega\) and \(r>0\), write
\[
M_h(\Gamma,r)
:=
\{x\in\overline\Omega:\operatorname{dist}_h(x,\Gamma)<r\}.
\]

For a positive \(a\in C^\infty(\overline\Omega)\), the metric governing travel times and the weight remain
\[
h=a^{-2}\mathrm{d}x^2,
\qquad
\varrho=a^{n-2}.
\]
Then \eqref{eq:density-wave-equation} becomes
\[
\Delta_{h,\varrho}=a^2\Delta,
\qquad
\mathrm{d}\mu_{h,\varrho}=a^{-2}\mathrm{d}x.
\]

The equality case below uses the reconstruction on the region of boundary depth less than \(R\).
We isolate the precise conclusion needed at the endpoint.

\begin{lemma}
\label{lem:finite-depth-weighted-reconstruction}
Let \(n\geq2\), let \(a_1,a_2\in C^\infty(\overline\Omega)\) be positive, and let \(R>0\).
For \(j=1,2\), set
\[
h_j=a_j^{-2}\mathrm{d}x^2,
\qquad
\varrho_j=a_j^{n-2},
\qquad
\mathrm{d}\mu_j=\varrho_j\mathrm{d}V_{h_j}=a_j^{-2}\mathrm{d}x,
\]
and define
\[
M_j^R
=
\{x\in\overline\Omega:\operatorname{dist}_{h_j}(x,\partial\Omega)<R\}.
\]
Assume that
\[
R_{a_1}^{2R}=R_{a_2}^{2R}
\]
and that \(\overline\Omega\setminus M_j^R\) has \(\mu_j\)-measure zero for \(j=1,2\).
Then there is a boundary-fixing smooth isometry
\[
\Phi^R:(M_1^R,h_1)\longrightarrow(M_2^R,h_2)
\]
such that
\[
(\Phi^R)^*\varrho_2=\varrho_1
\quad\text{on }M_1^R.
\]
Moreover, for every open \(\Gamma\subset\partial\Omega\) and \(0<r<R\),
\[
\Phi^R\bigl(M_{h_1}(\Gamma,r)\bigr)
=
M_{h_2}(\Gamma,r).
\]
\end{lemma}

\begin{proof}[Proof of Lemma~\ref{lem:finite-depth-weighted-reconstruction}]
Equations~\eqref{eq:density-wave-equation} and \eqref{eq:acoustic-flux-density} identify the scalar normal trace with the boundary response density.
Hence the two boundary response densities agree on \((0,2R)\).
The identity for terminal inner products in Step~1 of the proof of Proposition~\ref{prop:finite-time-spatial-recovery} therefore gives a common Gramian for waves at time \(R\).
Choose numbers \(r_k\uparrow R\).
For every \(k\), the restricted controllability statement \eqref{eq:weighted-restricted-controllability}, with \(\Gamma=\partial\Omega\) and \(r=r_k<R\), gives
\[
\overline{\mathscr W_j(\partial\Omega,r_k)}^{\,L^2(\mu_j)}
=
L^2(M_{h_j}(\partial\Omega,r_k),\mathrm{d}\mu_j).
\]
The sets on the right increase to \(M_j^R\), so their \(L^2\)-spaces have dense union in \(L^2(M_j^R,\mathrm{d}\mu_j)\).
The corresponding controls lie in \(C_c^\infty((0,R)\times\partial\Omega)\), while finite propagation places every terminal wave generated by such a control in \(L^2(M_j^R,\mathrm{d}\mu_j)\).
Consequently,
\[
\overline{
\{z_j^f(R):f\in C_c^\infty((0,R)\times\partial\Omega)\}
}^{\,L^2(\mu_j)}
=
L^2(M_j^R,\mathrm{d}\mu_j).
\]
The assumption that the complement is null identifies the space on the right with \(L^2(\Omega,\mathrm{d}\mu_j)\).
Consequently, the common terminal Gramian induces a unitary
\[
U:L^2(\Omega,\mathrm{d}\mu_1)\longrightarrow L^2(\Omega,\mathrm{d}\mu_2)
\]
that maps terminal waves generated by the same boundary source to one another.
For every open \(\Gamma\subset\partial\Omega\) and \(0<r<R\), restricted controllability then gives
\begin{equation}
U \Pi_{M_{h_1}(\Gamma,r)}U^{-1}
=
\Pi_{M_{h_2}(\Gamma,r)},
\label{eq:finite-depth-labelled-projections}
\end{equation}
where \(\Pi_A\) denotes multiplication by the characteristic function of \(A\).

We next pass from the labeled projections to the truncated boundary distance data.
For \(z\in\partial\Omega\) and \(0<s<R\), define
\[
\mathbb B_j(z,s)
=
\{x\in\overline\Omega:\operatorname{dist}_{h_j}(x,z)<s\}.
\]
Choose boundary neighborhoods \(\Gamma_m\) of \(z\) and numbers \(\epsilon_m\downarrow0\) such that \(\Gamma_m\) shrinks to \(z\) and
\[
\max_{j=1,2}\operatorname{diam}_{h_j}(\Gamma_m)<\frac{\epsilon_m}{2}.
\]
The triangle inequality gives
\[
\mathbb B_j(z,s)
=
\bigcup_{m:\,\epsilon_m<s}
M_{h_j}(\Gamma_m,s-\epsilon_m).
\]
Taking closed spans of the corresponding \(L^2\)-spaces and using \eqref{eq:finite-depth-labelled-projections} yields
\begin{equation}
U \Pi_{\mathbb B_1(z,s)}U^{-1}
=
\Pi_{\mathbb B_2(z,s)}.
\label{eq:finite-depth-ball-projections}
\end{equation}

Define the truncated boundary distance maps by
\[
\mathcal R_j^R(x)(z)
=
\min\{\operatorname{dist}_{h_j}(x,z),R\},
\qquad
x\in M_j^R.
\]
We prove that their images agree.
Fix \(x\in M_1^R\), set \(r_x=\mathcal R_1^R(x)\), and choose \(z_0\in\partial\Omega\) such that \(r_x(z_0)<R\).
Let \(\{z_\ell\}_{\ell=1}^\infty\) be dense in \(\partial\Omega\), with \(z_1=z_0\).
For each \(m\), choose \(0<\eta_m<1/m\) so small that
\[
r_x(z_0)+\eta_m<R
\]
and \(r_x(z_\ell)-\eta_m>0\), \(r_x(z_\ell)+\eta_m<R\) whenever \(1\leq\ell\leq m\) and \(0<r_x(z_\ell)<R\).
For brevity, write \(d_j(y,z)=\operatorname{dist}_{h_j}(y,z)\).
Define the finite distance cell
\[
\begin{aligned}
\mathcal E_{j,m}
={}&
\bigcap_{\substack{1\leq\ell\leq m\\ r_x(z_\ell)=0}}
\{y:d_j(y,z_\ell)<\eta_m\}
\\
&\cap
\bigcap_{\substack{1\leq\ell\leq m\\ 0<r_x(z_\ell)<R}}
\{y:|d_j(y,z_\ell)-r_x(z_\ell)|<\eta_m\}
\\
&\cap
\bigcap_{\substack{1\leq\ell\leq m\\ r_x(z_\ell)=R}}
\{y:d_j(y,z_\ell)>R-\eta_m\}.
\end{aligned}
\]
The projection onto \(L^2(\mathcal E_{j,m},\mathrm{d}\mu_j)\), modulo a null set, is a finite product of the ball projections in \eqref{eq:finite-depth-ball-projections} and their orthogonal complements.
Indeed, Lemma~\ref{lem:distance-level-null} shows that every positive distance sphere occurring at the boundary of the cell is null.
Conjugating this finite product by \(U\) and using \eqref{eq:finite-depth-ball-projections} shows that the two cell projections are simultaneously zero or nonzero.
The set \(\mathcal E_{1,m}\) contains \(x\) and is relatively open, so it has positive \(\mu_1\)-measure.
Thus \(\mathcal E_{2,m}\) is nonempty; choose \(x_m'\in\mathcal E_{2,m}\).
By compactness of \(\overline\Omega\), a subsequence converges to some \(x'\in\overline\Omega\).
The condition associated with \(z_0\) gives
\[
\operatorname{dist}_{h_2}(x',z_0)=r_x(z_0)<R,
\]
so \(x'\in M_2^R\).
For every fixed \(\ell\), the defining inequalities for \(\mathcal E_{2,m}\) and the limit \(\eta_m\to0\) give
\[
\mathcal R_2^R(x')(z_\ell)=r_x(z_\ell).
\]
Continuity in the boundary variable and density of \(\{z_\ell\}\) yield \(\mathcal R_2^R(x')=r_x\).
Interchanging the two metrics gives the reverse inclusion, and therefore
\begin{equation}
\mathcal R_1^R(M_1^R)
=
\mathcal R_2^R(M_2^R).
\label{eq:finite-depth-common-profile-set}
\end{equation}

The reconstruction in \cite[Theorem~6, the paragraphs surrounding formula~(74), and Lemma~6]{katchalov2004energy} shows that the truncated boundary distance map is a homeomorphism onto its image and that its evaluation functions determine the smooth structure and the metric.
Applying this result to the common image in \eqref{eq:finite-depth-common-profile-set} gives a boundary-fixing smooth isometry
\[
\Phi^R:(M_1^R,h_1)\longrightarrow(M_2^R,h_2).
\]
Indeed, the profile of \(z\in\partial\Omega\) has value zero at \(z\), so equality of profiles gives \(\Phi^R(z)=z\).
The labeled form of the construction and \eqref{eq:finite-depth-labelled-projections} give
\[
\Phi^R\bigl(M_{h_1}(\Gamma,r)\bigr)
=
M_{h_2}(\Gamma,r)
\]
for every relatively open \(\Gamma\subset\partial\Omega\) and \(0<r<R\).

It remains to recover the weight.
The identity for pairing with the constant function \eqref{eq:weighted-constant-pairing} and density of the terminal waves give
\[
U1_1=1_2,
\]
where \(1_j\) is the constant function one in \(L^2(\Omega,\mathrm{d}\mu_j)\).
Let \(\mathfrak A_R\) be the algebra generated by the sets \(M_{h_1}(\Gamma,r)\) with \(\Gamma\) open and \(0<r<R\).
Equation~\eqref{eq:finite-depth-labelled-projections} and finite Boolean operations imply
\[
U \Pi_AU^{-1}=\Pi_{\Phi^R(A)}
\quad\text{for every }A\in\mathfrak A_R.
\]
Hence
\[
\mu_2(\Phi^R(A))
=
\|\Pi_{\Phi^R(A)}1_2\|_{L^2(\mu_2)}^2
=
\|\Pi_A1_1\|_{L^2(\mu_1)}^2
=
\mu_1(A).
\]
The truncated boundary distance map is a homeomorphism and its evaluation functions separate points of \(M_1^R\).
Using a countable boundary basis and rational radii below \(R\), their rational sublevel sets belong to \(\sigma(\mathfrak A_R)\) and generate \(\mathcal B(M_1^R)\).
The \(\pi\)--\(\lambda\) theorem therefore extends the preceding measure identity to every Borel subset of \(M_1^R\).
Since \((\Phi^R)^*h_2=h_1\), comparison of
\[
\mathrm{d}\mu_j=\varrho_j\mathrm{d}V_{h_j}
\]
gives
\[
(\Phi^R)^*\varrho_2=\varrho_1
\quad\text{on }M_1^R.
\]

The proof is complete.
\end{proof}

\begin{proposition}
\label{prop:time-optimal-density}
Let \(S_T>0\) and \(n\geq2\), and let \(a_1,a_2\in C^\infty(\overline\Omega)\) be positive.
Set \(h_j=a_j^{-2}\mathrm{d}x^2\), \(j=1,2\).
If
\begin{equation}
R_{a_1}^{S_T}=R_{a_2}^{S_T},
\qquad
S_T\geq2\max\{\operatorname{rad}(h_1),\operatorname{rad}(h_2)\},
\label{eq:bc-global-hypothesis}
\end{equation}
then
\[
a_1=a_2
\quad\text{on }\overline\Omega.
\]
\end{proposition}

\begin{proof}
For \(j=1,2\), set \(\varrho_j=a_j^{n-2}\), and put
\[
R=\frac{S_T}{2},
\qquad
M_j^R
=
\{x\in\overline\Omega:\operatorname{dist}_{h_j}(x,\partial\Omega)<R\},
\qquad
C_j
=
\{x\in\overline\Omega:\operatorname{dist}_{h_j}(x,\partial\Omega)=R\}.
\]
We first treat the endpoint
\begin{equation}
R=\max\{\operatorname{rad}(h_1),\operatorname{rad}(h_2)\}.
\label{eq:bc-endpoint-case}
\end{equation}
Lemma~\ref{lem:distance-level-null} shows that \(C_j\) has zero Riemannian volume with respect to \(h_j\) for \(j=1,2\).
Since \(\varrho_j\) is smooth and positive, \(C_j\) is also null for the weighted measure in Lemma~\ref{lem:finite-depth-weighted-reconstruction}.
Moreover, \eqref{eq:bc-endpoint-case} gives \(\operatorname{rad}(h_j)\leq R\), and hence
\[
\overline\Omega\setminus M_j^R=C_j.
\]
Thus the measure zero hypothesis of Lemma~\ref{lem:finite-depth-weighted-reconstruction} is satisfied.
Since \(S_T=2R\), that lemma gives a boundary-fixing smooth isometry
\begin{equation}
\Phi^R:(M_1^R,h_1)\longrightarrow(M_2^R,h_2)
\label{eq:bc-endpoint-open-isometry}
\end{equation}
satisfying
\begin{equation}
(\Phi^R)^*\varrho_2=\varrho_1
\quad\text{on }M_1^R.
\label{eq:bc-endpoint-weight-identity}
\end{equation}

We next show that \(C_1\) and \(C_2\) are either both empty or both nonempty.
If \(C_1=\varnothing\), then \(M_1^R=\overline\Omega\) is compact.
If \(C_2\neq\varnothing\), then its nullity implies that it has empty interior, so \(M_2^R=\overline\Omega\setminus C_2\) is a proper nonclosed subset of the compact space \(\overline\Omega\).
Thus \(M_2^R\) is not compact, contradicting the homeomorphism \(\Phi^R:M_1^R\to M_2^R\).
Interchanging the indices gives the converse.
Condition~\eqref{eq:bc-endpoint-case} ensures that at least one \(C_j\) is nonempty.
Therefore both sets are nonempty and
\[
R=\operatorname{rad}(h_1)=\operatorname{rad}(h_2).
\]

Every connected component \(O\) of \(M_1^R\) contains an open portion of \(\partial\Omega\).
Indeed, compactness provides a minimizing geodesic for \(h_1\) from any \(x\in O\) to a closest boundary point.
It remains in \(M_1^R\), because the distance to \(\partial\Omega\) along the geodesic is strictly smaller than \(R\).
The component \(O\) therefore meets \(\partial\Omega\), and relative openness of \(M_1^R\) supplies an open boundary portion contained in \(O\).

If \(n\geq3\), equation~\eqref{eq:bc-endpoint-weight-identity} gives \(a_2\circ\Phi^R=a_1\).
Combining this identity with \((\Phi^R)^*h_2=h_1\) gives
\[
(\Phi^R)^*(\mathrm{d}x^2)=\mathrm{d}x^2.
\]
Thus \(\Phi^R|_O\) is the restriction of a Euclidean rigid motion.
On the open boundary portion contained in \(O\), its tangential differential is the identity.
Since \(\Phi^R\) maps the inward side of \(\Omega\) to itself, it also fixes the inward unit normal.
Hence its full differential is the identity at a boundary point, and the rigid motion is the identity on \(O\).

If \(n=2\), the metric identity makes \(\Phi^R|_O\) conformal with respect to the Euclidean metric.
Boundary fixing and preservation of the inward side make it orientation preserving.
On the open boundary portion contained in \(O\), its tangential derivative is the identity, and conformality together with preservation of the inward side gives the same conclusion for its normal derivative.
An orientation-preserving conformal map between planar domains is locally holomorphic, so each Euclidean coordinate of \(\Phi^R\) is harmonic.
Consequently, every component of \(\Phi^R-\operatorname{Id}\) is harmonic and has vanishing Dirichlet and Neumann data on that boundary portion.
Boundary unique continuation for the Laplacian, applied to each component of \(\Phi^R-\operatorname{Id}\), gives \(\Phi^R=\operatorname{Id}\) on \(O\).

In every dimension \(n\geq2\), the metric identity now yields
\[
a_1=a_2
\quad\text{on }M_1^R.
\]
The set \(M_1^R\) is dense in \(\overline\Omega\) because \(C_1\) is null.
Continuity therefore extends the equality to \(\overline\Omega\).
This proves the proposition in the endpoint case.

It remains to consider
\[
R>\max\{\operatorname{rad}(h_1),\operatorname{rad}(h_2)\}.
\]
Equivalently,
\[
S_T>2\max\{\operatorname{rad}(h_1),\operatorname{rad}(h_2)\}.
\]
For \(j=1,2\), equations~\eqref{eq:density-wave-equation} and \eqref{eq:acoustic-flux-density} give
\[
\mathcal B_{h_j,\varrho_j}^{S_T}f
=
(R_{a_j}^{S_T}f)\,\mathrm{d}S_{\mathrm{d}x^2}.
\]
Thus \(R_{a_1}^{S_T}=R_{a_2}^{S_T}\) yields \eqref{eq:weighted-response-equality} with \(M_1=M_2=\overline\Omega\) and \(\varphi=\operatorname{Id}_{\partial\Omega}\).
Proposition~\ref{prop:finite-time-spatial-recovery} therefore gives a diffeomorphism \(\Phi:\overline\Omega\to\overline\Omega\) such that
\begin{equation}
\Phi|_{\partial\Omega}=\operatorname{Id},
\qquad
\Phi^*h_2=h_1,
\qquad
\Phi^*\varrho_2=\varrho_1.
\label{eq:acoustic-weighted-isometry}
\end{equation}
For each \(\ell=1,\ldots,n\), one has
\[
\Delta_{h_j,\varrho_j}x^\ell
=
a_j^2\Delta x^\ell
=
0.
\]
The weighted isometry in \eqref{eq:acoustic-weighted-isometry} intertwines the two weighted Laplacians, so \(x^\ell\circ\Phi\) is harmonic with respect to \(\Delta_{h_1,\varrho_1}\).
Since \(\Phi\) fixes the boundary, \(x^\ell\circ\Phi\) and \(x^\ell\) have the same Dirichlet boundary values.
Uniqueness for the uniformly elliptic Dirichlet problem gives
\[
x^\ell\circ\Phi=x^\ell
\quad\text{in }\Omega.
\]
Thus \(\Phi=\operatorname{Id}\).
With \(\Phi=\operatorname{Id}\), the metric identity in \eqref{eq:acoustic-weighted-isometry} becomes
\[
a_2^{-2}\mathrm{d}x^2=a_1^{-2}\mathrm{d}x^2,
\]
and positivity yields \(a_1=a_2\).

The proof is complete.
\end{proof}

\begin{proof}[Proof of Corollary~\ref{cor:fixed-euclidean-recovery}]
Let \(c_j\) and their normalized factors \(a_j,b_j\), \(j=1,2\), satisfy the hypotheses, and suppose that
\[
\Lambda_{c_1}^T
=
\Lambda_{c_2}^T.
\]
Proposition~\ref{prop:boundary-data-correspondence} gives \(b_1=b_2=:b\) on \([0,T]\) and
\[
R_{a_1}^{S_T}
=
R_{a_2}^{S_T},
\qquad
S_T=\int_0^T b(t)\,\mathrm{d}t.
\]
The visibility assumptions give
\[
S_T
\geq
2\max\{\operatorname{rad}(h_{a_1}),\operatorname{rad}(h_{a_2})\}.
\]
The operators \(R_{a_j}^{S_T}\) therefore satisfy the hypotheses of Proposition~\ref{prop:time-optimal-density}.
Proposition~\ref{prop:time-optimal-density} gives \(a_1=a_2\) on \(\overline\Omega\).
Combining this equality with \(b_1=b_2=b\) gives \(c_1(x,t)=a_1(x)b(t)=a_2(x)b(t)=c_2(x,t)\) on \(\overline\Omega\times[0,T]\).
\end{proof}

\subsection{A uniform Euclidean visibility condition}

If \(a_j\geq a_*>0\), then
\[
\operatorname{rad}(h_{a_j})\leq \frac{r_\Omega}{a_*}.
\]
Hence the following uniform condition implies the visibility hypothesis for both models.

\begin{corollary}
\label{cor:uniform-finite-time}
Under the assumptions of Corollary~\ref{cor:fixed-euclidean-recovery}, assume in addition that \(a_j\geq a_*>0\) on \(\overline\Omega\).
If
\begin{equation}
\int_0^T b_j(t)\,\mathrm{d}t
\geq
\frac{2r_\Omega}{a_*},
\qquad
j=1,2,
\label{eq:uniform-time-condition}
\end{equation}
then equality of the two finite-time Dirichlet-to-Neumann maps implies
\[
c_1=c_2
\quad\text{on }\overline\Omega\times[0,T].
\]
\end{corollary}

\begin{proof}[Proof of Corollary~\ref{cor:uniform-finite-time}]
Fix \(x\in\Omega\), and choose \(y\in\partial\Omega\) with
\[
|x-y|=\operatorname{dist}_{\mathrm{d}x^2}(x,\partial\Omega).
\]
The segment from \(x\) to \(y\) lies in \(\overline\Omega\); otherwise, its first exit point would be a boundary point strictly closer to \(x\).
Since \(a_j\geq a_*\), the length of this segment with respect to \(h_{a_j}\) is at most
\[
\frac{\operatorname{dist}_{\mathrm{d}x^2}(x,\partial\Omega)}{a_*}.
\]
Consequently,
\begin{equation}
\operatorname{rad}(h_{a_j})
\leq
\frac{r_\Omega}{a_*},
\qquad
j=1,2.
\label{eq:uniform-radius-bound}
\end{equation}
Equations~\eqref{eq:uniform-time-condition} and \eqref{eq:uniform-radius-bound} give
\[
\int_0^T b_j(t)\,\mathrm{d}t
\geq
2\operatorname{rad}(h_{a_j}).
\]
Corollary~\ref{cor:fixed-euclidean-recovery} then yields \(c_1=c_2\) on \(\overline\Omega\times[0,T]\).
\end{proof}

\subsection{Sharpness of the uniform Euclidean observation scale}
\label{sec:causal-obstruction}

Continue in the fixed Euclidean setting and retain the abbreviations \(\Lambda_c^T\) and \(R_a^{S_T}\).
The boundary source must first reach an interior perturbation, and the resulting discrepancy must then return to the boundary before it can be measured.
We estimate these two travel times separately and then construct a perturbation that is invisible before their sum.
This proves the optimality of the constant \(2\) in \eqref{eq:uniform-time-condition} and yields distinct models whose Dirichlet-to-Neumann maps agree for every finite \(T>0\) when the total effective time is finite.

\subsubsection{Propagation to an interior perturbation and back}

Let \(\mathbb B(x,r)\) denote the open Euclidean ball with center \(x\) and radius \(r\).
Figure~\ref{fig:causal-roundtrip} illustrates the two successive stages of propagation.

\begin{figure}[htbp]
\centering
\includegraphics[width=0.72\textwidth]{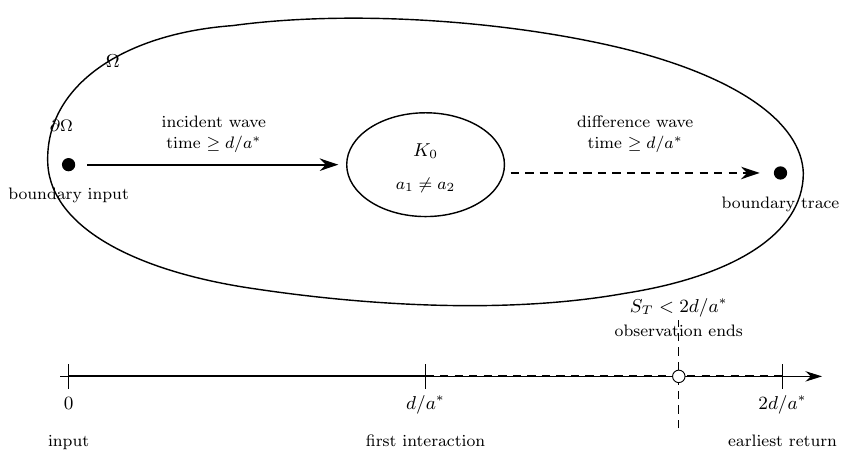}
\caption{The causal mechanism: a wave generated at the boundary must first reach the perturbed region, and the resulting discrepancy must then return to the boundary.
Under the bound \(a(x)\leq a^*\), each part requires at least \(d/a^*\), so the perturbation cannot affect the boundary response before \(2d/a^*\); in Proposition~\ref{prop:causal-obstruction}, \(a^*=\varepsilon\) and the accumulated effective time is \(S_T\).}
\label{fig:causal-roundtrip}
\end{figure}

The following lemma records the corresponding estimates under a common upper speed bound.

\begin{lemma}
\label{lem:finite-propagation}
Let \(S_T>0\), let \(a^*>0\), and let \(a\in C^\infty(\overline\Omega)\) satisfy
\begin{equation}
0<a(x)\leq a^*
\quad\text{on }\overline\Omega.
\label{eq:speed-upper-bound}
\end{equation}
\begin{enumerate}
\item[(a)]
Let \(f\in\mathcal D_{S_T}\), and let \(z\) solve \(z_{ss}-a(x)^2\Delta z=0\) on \(\Omega\times(0,S_T)\) with zero initial data and Dirichlet boundary value \(f\).
Then
\begin{equation}
z(x,s)=0
\quad\text{whenever}\quad
0<s<\min\left\{S_T,\frac{\operatorname{dist}_{\mathrm{d}x^2}(x,\partial\Omega)}{a^*}\right\}.
\label{eq:boundary-finite-propagation}
\end{equation}
\item[(b)]
Let \(K_0\Subset\Omega\), put \(d=\operatorname{dist}_{\mathrm{d}x^2}(K_0,\partial\Omega)>0\), and suppose that \(w\) has zero initial data, zero Dirichlet boundary data, and satisfies on \(\Omega\times(0,S_T)\)
\[
w_{ss}-a(x)^2\Delta w=F,
\qquad
\operatorname{supp}F
\subset
K_0\times\bigl([s_0,\infty)\cap(0,S_T)\bigr),
\]
where \(s_0\geq0\).
Then, for \(s_0\leq s<S_T\),
\begin{equation}
\operatorname{supp}w(\cdot,s)
\subset
\{x\in\overline\Omega:
\operatorname{dist}_{\mathrm{d}x^2}(x,K_0)
\leq
a^*(s-s_0)\}.
\label{eq:source-finite-propagation}
\end{equation}
In particular,
\begin{equation}
\partial_\nu w|_{\partial\Omega}=0
\quad\text{for}\quad
0<s<\min\left\{S_T,s_0+\frac{d}{a^*}\right\}.
\label{eq:return-time-bound}
\end{equation}
\end{enumerate}
\end{lemma}

\begin{proof}
We first establish an energy estimate on a backward cone that will be used for both parts.
Let \(u\) satisfy
\[
u_{ss}-a(x)^2\Delta u=G.
\]
After multiplication by \(a^{-2}\), this equation becomes
\[
a^{-2}u_{ss}-\Delta u=a^{-2}G.
\]
Define
\begin{equation}
e(u)
=
\frac12\bigl(a^{-2}|u_s|^2+|\nabla u|^2\bigr).
\label{eq:energy-density}
\end{equation}
Since \(a\) is independent of \(s\),
\begin{equation}
\partial_se(u)
=
\nabla\cdot(u_s\nabla u)
+
a^{-2}Gu_s.
\label{eq:local-energy-identity}
\end{equation}
Fix \(0\leq s_1<s_2<S_T\), \(x_0\in\Omega\), and \(r>0\), and consider the backward cone with sections
\[
D(s)
=
\Omega\cap
\mathbb B\bigl(x_0,r+a^*(s_2-s)\bigr),
\qquad
s_1\leq s\leq s_2.
\]
Set
\[
E(s)=\int_{D(s)}e(u)(x,s)\,\mathrm{d}x.
\]
The spherical part of \(\partial D(s)\) moves inward with normal velocity \(-a^*\), whereas the part contained in \(\partial\Omega\) is fixed.
The form of Reynolds' transport theorem for moving volumes \cite{lidstrom2011moving}, together with \eqref{eq:local-energy-identity} and the divergence theorem, gives
\begin{align}
E'(s)
={}&
\int_{\partial D(s)\cap\Omega}
\bigl(u_s\partial_n u-a^* e(u)\bigr)\,\mathrm{d}S
+
\int_{D(s)}a^{-2}Gu_s\,\mathrm{d}x
\notag\\
&+
\int_{\partial\Omega\cap \mathbb B(x_0,r+a^*(s_2-s))}
u_s\partial_\nu u\,\mathrm{d}S.
\label{eq:moving-energy}
\end{align}
The two boundary pieces meet along a set of codimension two, which contributes no surface integral; alternatively, one may obtain \eqref{eq:moving-energy} by a smooth approximation of \(D(s)\).

Suppose that the cone does not meet \(\operatorname{supp}G\), and that either it stays inside \(\Omega\) or \(u\) has zero Dirichlet data on the portion of \(\partial\Omega\) that it meets.
Then the volume source term vanishes.
In the second boundary case, \(u_s=0\) on \(\partial\Omega\), so the last term in \eqref{eq:moving-energy} also vanishes.
On the moving spherical part, \eqref{eq:speed-upper-bound} implies \(a^*a^{-2}\geq(a^*)^{-1}\), and hence
\begin{align*}
u_s\partial_n u-a^* e(u)
&\leq
|u_s||\nabla u|
-\frac{1}{2a^*}|u_s|^2
-\frac{a^*}{2}|\nabla u|^2\\
&\leq0.
\end{align*}
Consequently,
\begin{equation}
E(s_2)\leq E(s_1).
\label{eq:backward-cone-energy}
\end{equation}
This is the common cone estimate.
If \(u=u_s=0\) on \(D(s_1)\), then \(E(s_1)=0\).
Since \(e(u)\geq0\) and the same computation gives \(E'(s)\leq0\) throughout \([s_1,s_2]\), it follows that \(E(s)=0\) throughout the cone.
Moreover, \(D(s_2)\subset D(s)\) for \(s_1\leq s\leq s_2\).
Thus, for \(x\in D(s_2)\), one has \(u_s(x,s)=0\) along the vertical segment from \(s_1\) to \(s_2\), and hence
\[
u(x,s_2)
=
u(x,s_1)+\int_{s_1}^{s_2}u_s(x,s)\,\mathrm{d}s
=
0.
\]
Therefore, vanishing Cauchy data at the base of a backward cone that contains no source imply that \(u\) vanishes at its top.

We now apply this estimate to part~(a).
Fix \((x_0,s_2)\) such that
\[
a^*s_2<\operatorname{dist}_{\mathrm{d}x^2}(x_0,\partial\Omega).
\]
Choose \(r>0\) with
\[
r+a^*s_2<\operatorname{dist}_{\mathrm{d}x^2}(x_0,\partial\Omega).
\]
For \(u=z\), \(G=0\), and \(s_1=0\), the entire backward cone lies in the interior of \(\Omega\), and its base has zero Cauchy data.
The cone estimate gives \(z=0\) in \(\mathbb B(x_0,r)\) at time \(s_2\), and in particular \(z(x_0,s_2)=0\).
This proves \eqref{eq:boundary-finite-propagation}.

For part~(b), the assertion is immediate if \(s_0\geq S_T\), so assume \(s_0<S_T\).
Since \(F=0\) for \(0<s<s_0\) and \(w\) has zero initial and Dirichlet data, uniqueness gives
\[
w=0
\quad\text{for}\quad
0\leq s\leq s_0;
\]
in particular, \(w(\cdot,s_0)=w_s(\cdot,s_0)=0\).

Fix \(s_0\leq s_2<S_T\) and a point \(x_0\) satisfying
\[
\operatorname{dist}_{\mathrm{d}x^2}(x_0,K_0)
>
a^*(s_2-s_0).
\]
Choose \(r>0\) sufficiently small that
\[
\operatorname{dist}_{\mathrm{d}x^2}(\mathbb B(x_0,r),K_0)
>
a^*(s_2-s_0).
\]
Indeed, for \(s_0\leq s\leq s_2\), enlarging \(\mathbb B(x_0,r)\) by \(a^*(s_2-s)\) still leaves its distance from \(K_0\) greater than
\[
a^*(s_2-s_0)-a^*(s_2-s)
=
a^*(s-s_0)
\geq0.
\]
Thus the backward cone over \([s_0,s_2]\) avoids \(K_0\times[s_0,s_2]\), and hence avoids \(\operatorname{supp}F\).
Its base has zero Cauchy data, while its physical boundary portion has zero Dirichlet data.
Applying the cone estimate with \(u=w\), \(G=F\), and \(s_1=s_0\) yields \(w(x_0,s_2)=0\).
Since \(x_0\) was arbitrary, this proves \eqref{eq:source-finite-propagation}.

It remains only to extract the boundary consequence.
If \(s_0\leq s<\min\{S_T,s_0+d/a^*\}\), the set on the right-hand side of \eqref{eq:source-finite-propagation} has positive distance from \(\partial\Omega\).
More precisely, this distance is at least \(d-a^*(s-s_0)>0\).
Thus \(w(\cdot,s)\) vanishes in a full boundary collar and \(\partial_\nu w=0\) on \(\partial\Omega\).
Together with the already established vanishing for \(0<s<s_0\), this proves \eqref{eq:return-time-bound}.

The proof is complete.
\end{proof}

\subsubsection{Perturbations invisible before the return time}

The next proposition applies Lemma~\ref{lem:finite-propagation} to \eqref{eq:stationary-forward} and shows that a perturbation supported in \(K_0\Subset\Omega\) is invisible until the return time.

\begin{proposition}
\label{prop:causal-obstruction}
Let \(K_0\Subset\Omega\) be compact with nonempty interior, let \(b\in C^\infty([0,T])\) be positive with \(b(0)=1\), and set \(S_T=\int_0^T b(t)\,\mathrm{d}t\).
Fix \(\varepsilon>0\), choose
\[
0\neq\chi\in C_c^\infty(\operatorname{int}K_0),
\qquad
0\leq\chi\leq1,
\qquad
0<\delta<1,
\]
and define
\begin{equation}
a_1(x)=\varepsilon,
\qquad
a_2(x)=\varepsilon\bigl(1-\delta\chi(x)\bigr).
\label{eq:invisible-profiles}
\end{equation}
If
\[
S_T
\leq
\frac{2\operatorname{dist}_{\mathrm{d}x^2}(K_0,\partial\Omega)}{\varepsilon}
\]
then the wave speeds \(c_j=a_jb\) satisfy
\[
\Lambda_{c_1}^T=\Lambda_{c_2}^T.
\]
\end{proposition}

\begin{proof}
Set
\begin{equation}
d
=
\operatorname{dist}_{\mathrm{d}x^2}(K_0,\partial\Omega)
>0.
\label{eq:perturbation-depth}
\end{equation}
Write
\begin{equation}
S_T=\int_0^T b(t)\,\mathrm{d}t.
\label{eq:counterexample-effective-time}
\end{equation}
The hypothesis, including the endpoint, is
\begin{equation}
S_T\leq\frac{2d}{\varepsilon}.
\label{eq:counterexample-threshold}
\end{equation}
The profiles in \eqref{eq:invisible-profiles} satisfy
\begin{equation}
0<a_2\leq a_1,
\qquad
a_2\not\equiv a_1,
\qquad
\operatorname{supp}(a_2-a_1)\subset K_0.
\label{eq:invisible-profile-properties}
\end{equation}
In particular, \(a_1=a_2\) in a neighborhood of \(\partial\Omega\).

To prove that this perturbation is invisible at the boundary before the return time, let \(f\) be admissible smooth boundary data on \(\partial\Omega\times(0,S_T)\), and let \(z_j\) solve
\begin{equation}
\begin{cases}
(z_j)_{ss}-a_j(x)^2\Delta z_j=0
&\text{in }\Omega\times(0,S_T),\\
z_j|_{s=0}=(z_j)_s|_{s=0}=0
&\text{in }\Omega,\\
z_j=f
&\text{on }\partial\Omega\times(0,S_T).
\end{cases}
\label{eq:counterexample-stationary-solutions}
\end{equation}
Since \(a_1\equiv\varepsilon\), Lemma~\ref{lem:finite-propagation}(a) gives
\begin{equation}
z_1=0
\quad\text{on}\quad
K_0\times\left(0,\frac d\varepsilon\right).
\label{eq:baseline-not-reached}
\end{equation}
For each point in this set the vanishing holds in a spacetime neighborhood, and hence
\begin{equation}
\Delta z_1=0
\quad\text{on}\quad
K_0\times\left(0,\frac d\varepsilon\right).
\label{eq:baseline-laplacian-zero}
\end{equation}

Set \(w=z_2-z_1\).
Then \(w\) has zero initial and Dirichlet boundary data, and
\begin{align}
w_{ss}-a_2^2\Delta w
&=
(a_2^2-a_1^2)\Delta z_1
\notag\\
&=:F.
\label{eq:difference-source}
\end{align}
The support property in \eqref{eq:invisible-profile-properties} and \eqref{eq:baseline-laplacian-zero} imply, relative to the finite cylinder \(\Omega\times(0,S_T)\),
\begin{equation}
\operatorname{supp}F
\subset
K_0\times
\left(
\left[\frac d\varepsilon,\infty\right)
\cap(0,S_T)
\right).
\label{eq:difference-source-support}
\end{equation}
Moreover, \(a_2\leq\varepsilon\).
Applying Lemma~\ref{lem:finite-propagation}(b) with \(a^*=\varepsilon\) and \(s_0=d/\varepsilon\) gives
\begin{equation}
\partial_\nu w|_{\partial\Omega}=0
\quad\text{for}\quad
0<s<\min\left\{S_T,\frac{2d}{\varepsilon}\right\}.
\label{eq:no-return-before-roundtrip}
\end{equation}
Condition \eqref{eq:counterexample-threshold} therefore yields
\[
\partial_\nu z_1=\partial_\nu z_2
\quad\text{on }\partial\Omega\times(0,S_T).
\]
Since \(f\) was arbitrary,
\begin{equation}
R_{a_1}^{S_T}=R_{a_2}^{S_T}.
\label{eq:equal-invisible-stationary}
\end{equation}

To transfer \eqref{eq:equal-invisible-stationary} to the measured responses, set \(c_j=a_jb\).
Since both wave speeds have temporal factor \(b\), they share \(B\), \(\beta\), \(U_b^{S_T}\), \(q\), and \(T_q^{S_T}\).
The reverse operator identity \eqref{eq:physical-from-stationary} gives
\[
\Lambda_{c_j}^T
=
(U_b^{S_T})^{-1}
T_q^{S_T}R_{a_j}^{S_T}
(T_q^{S_T})^{-1}
U_b^{S_T}.
\]
Substitution of \eqref{eq:equal-invisible-stationary} into the right-hand side gives
\begin{equation}
\Lambda_{c_1}^T=\Lambda_{c_2}^T.
\label{eq:equal-invisible-physical}
\end{equation}

The proof is complete.
\end{proof}

\subsubsection{Optimality of the factor \(2\)}

Placing the support near a point of maximal Euclidean depth turns Proposition~\ref{prop:causal-obstruction} into the following optimality result: no \(C<2\) can replace the uniform factor \(2\).

\begin{theorem}
\label{thm:causal-obstruction}
Let \(n\geq2\), let \(\Omega\subset\mathbb R^n\) be a smooth bounded connected domain, and write
\[
r_\Omega
=
\max_{x\in\overline\Omega}
\operatorname{dist}_{\mathrm{d}x^2}(x,\partial\Omega).
\]
For every \(0<C<2\) and \(T>0\), there exist a number \(a_*>0\), strictly positive functions \(a_1,a_2\in C^\infty(\overline\Omega)\), and a strictly positive nonconstant function \(b\in C^\infty([0,T])\) with \(b(0)=1\) such that the wave speeds
\[
c_j(x,t)=a_j(x)b(t),
\qquad
j=1,2,
\]
are distinct and satisfy
\[
a_j\geq a_*
\quad\text{on }\overline\Omega,
\qquad
j=1,2,
\]
and
\begin{equation}
\int_0^T b(t)\,\mathrm{d}t
>
\frac{C r_\Omega}{a_*},
\label{eq:subcritical-uniform-constant}
\end{equation}
but
\[
\Lambda_{c_1}^T=\Lambda_{c_2}^T.
\]
Consequently, no universal constant smaller than \(2\) can replace \(2\) in the Euclidean observation condition based on \(r_\Omega\) and a common lower bound for the spatial factors.
\end{theorem}

\begin{proof}[Proof of Theorem~\ref{thm:causal-obstruction}]
Choose \(x_0\in\Omega\) with
\[
\operatorname{dist}_{\mathrm{d}x^2}(x_0,\partial\Omega)
=
r_\Omega.
\]
Since \(C<2\), choose \(0<r<r_\Omega\) and \(0<\delta<1\) sufficiently small that
\begin{equation}
\frac{C}{1-\delta}
<
2\left(1-\frac r{r_\Omega}\right).
\label{eq:sharp-parameter-choice}
\end{equation}
Set
\[
K_0=\overline{\mathbb B(x_0,r)}\Subset\Omega.
\]
Then
\[
d
:=
\operatorname{dist}_{\mathrm{d}x^2}(K_0,\partial\Omega)
=
r_\Omega-r.
\]
Choose a smooth positive nonconstant temporal factor \(b\in C^\infty([0,T])\) with \(b(0)=1\), and write
\[
S_T=\int_0^T b(t)\,\mathrm{d}t.
\]
Since \(d=r_\Omega-r\), condition \eqref{eq:sharp-parameter-choice} is exactly
\[
\frac{C r_\Omega}{S_T(1-\delta)}
<
\frac{2d}{S_T}.
\]
We may therefore choose
\[
\frac{C r_\Omega}{S_T(1-\delta)}
<
\varepsilon
<
\frac{2d}{S_T}
\]
and take a nonzero \(\chi\in C_c^\infty(\operatorname{int}K_0)\) with \(0\leq\chi\leq1\).
Define
\[
a_1=\varepsilon,
\qquad
a_2=\varepsilon(1-\delta\chi),
\qquad
a_*=\varepsilon(1-\delta).
\]
Set \(c_j=a_jb\).
Since \(0\leq\chi\leq1\) and \(\chi\not\equiv0\), we have \(a_j\geq a_*>0\) and \(c_1\not\equiv c_2\).
The lower bound on \(\varepsilon\) gives
\[
S_T
>
\frac{C r_\Omega}{\varepsilon(1-\delta)}
=
\frac{C r_\Omega}{a_*},
\]
which is \eqref{eq:subcritical-uniform-constant}.
The upper bound on \(\varepsilon\) gives
\[
S_T<\frac{2d}{\varepsilon}.
\]
Moreover, \(a_1=\varepsilon\), \(0<a_2\leq\varepsilon\), and \(a_1-a_2\) is supported in \(K_0\).
Proposition~\ref{prop:causal-obstruction} therefore gives
\[
\Lambda_{c_1}^T=\Lambda_{c_2}^T.
\]
Thus every \(C<2\) fails as a universal replacement for the constant \(2\).

The proof is complete.
\end{proof}

\subsubsection{Infinite physical time with finite total effective time}

Letting the physical horizon tend to infinity while keeping the total effective time below the time required to reach the perturbation and return to the boundary gives the following limiting obstruction.

\begin{corollary}
\label{cor:infinite-time-threshold}
Let \(n\geq2\), and let \(\Omega\subset\mathbb R^n\) be a smooth bounded connected domain.
There exist strictly positive functions \(a_1,a_2\in C^\infty(\overline\Omega)\) and \(b\in C^\infty([0,\infty))\), with \(b(0)=1\), such that the wave speeds
\[
c_j(x,t)=a_j(x)b(t),
\qquad
j=1,2,
\]
are distinct and satisfy
\[
\Lambda_{c_1}^T
=
\Lambda_{c_2}^T
\qquad
\text{for every }T>0.
\]
The temporal factor can be chosen so that
\[
\int_0^\infty b(t)\,\mathrm{d}t<\infty.
\]
The spatial factors can simultaneously be chosen to satisfy
\[
a_1=a_2\text{ near }\partial\Omega,
\qquad
0<a_2\leq a_1\text{ on }\overline\Omega,
\qquad
a_1\not\equiv a_2.
\]
One may take \(b(t)=e^{-Mt}\) with \(M\) sufficiently large.
\end{corollary}

\begin{proof}[Proof of Corollary~\ref{cor:infinite-time-threshold}]
Choose a compact set \(K_0\Subset\Omega\) with nonempty interior and set
\[
d=\operatorname{dist}_{\mathrm{d}x^2}(K_0,\partial\Omega)>0.
\]
Fix \(\varepsilon>0\), choose a nonzero \(\chi\in C_c^\infty(\operatorname{int}K_0)\) with \(0\leq\chi\leq1\), and fix \(0<\delta<1\).
Define
\[
a_1=\varepsilon,
\qquad
a_2=\varepsilon(1-\delta\chi).
\]
Choose \(M>0\) so that \(1/M\leq2d/\varepsilon\), and set \(b(t)=e^{-Mt}\) and \(c_j(x,t)=a_j(x)b(t)\).
For every finite \(T>0\),
\[
S_T
=
\int_0^T b(t)\,\mathrm{d}t
=
\frac{1-e^{-MT}}{M}
<
\frac1M
\leq
\frac{2d}{\varepsilon}.
\]
Proposition~\ref{prop:causal-obstruction} therefore gives \(\Lambda_{c_1}^T=\Lambda_{c_2}^T\) for every \(T>0\).
Since \(\chi\) is supported in the interior of \(K_0\), we have \(a_1=a_2\) near \(\partial\Omega\), while \(0<a_2\leq a_1\) and \(a_1\not\equiv a_2\) follow from \(0<\delta<1\) and \(\chi\not\equiv0\).
The choice \(b(t)=e^{-Mt}\) is positive, satisfies \(b(0)=1\), and has total effective time \(1/M<\infty\).
\end{proof}

\begin{remark}
Suppose \(n\geq3\) and that two models with a common temporal factor satisfy \eqref{eq:boundary-identification}, for the same \(\varphi\), for every \(T>0\).
If
\[
\int_0^\infty b(t)\,\mathrm{d}t
>
2\max\{\operatorname{rad}(h_1),\operatorname{rad}(h_2)\},
\qquad
h_j=a_j^{-2}g_j,
\]
then the same inequality holds with \(S_{T_0}\) for some finite \(T_0\), and Theorem~\ref{thm:simultaneous-recovery}, applied at \(T_0\), gives simultaneous recovery.
\end{remark}

The corollary and remark confirm that long-time recovery is governed by total effective time.

\subsection{Failure at the critical endpoint on unknown manifolds}
\label{sec:endpoint-observation}

The subcritical perturbations above do not address the critical endpoint, where equality gives recovery on a fixed Euclidean background.
On an unknown manifold, the deepest hypersurface can instead carry a gluing ambiguity whose first boundary return occurs only at the excluded endpoint \(2R\).
The following theorem realizes this ambiguity even for the time-independent unit wave speed.

\begin{theorem}
\label{thm:endpoint-gluing-obstruction}
For every \(n\geq3\) and \(R>0\), there exist smooth compact connected \(n\)-dimensional Riemannian manifolds \((M_1,g_1)\) and \((M_2,g_2)\), each with smooth nonempty boundary, and a boundary diffeomorphism \(\varphi:\partial M_1\to\partial M_2\) such that
\[
\operatorname{rad}(g_1)=\operatorname{rad}(g_2)=R,
\]
and the full boundary Dirichlet-to-Neumann maps for the unit wave speed satisfy
\[
\Lambda_{g_1,1}^{2R}
=
P_\varphi^{-1}\Lambda_{g_2,1}^{2R}P_\varphi
\]
on the open observation interval \((0,2R)\), where \((P_\varphi f)(z,t)=f(\varphi^{-1}(z),t)\).
However, there is no isometry \(F:(M_1,g_1)\to(M_2,g_2)\) satisfying \(F|_{\partial M_1}=\varphi\).
\end{theorem}

Here the cut pieces and their local metrics agree; only the deepest identification differs.
The following lemma shows that finite propagation makes this change invisible.

\begin{lemma}
\label{lem:endpoint-seam-propagation}
Let \((N_\pm,g_\pm)\) be compact Riemannian manifolds with
\[
\partial N_\pm=B_\pm\dot\cup H_\pm,
\qquad
\operatorname{dist}_{g_\pm}(B_\pm,H_\pm)\geq R.
\]
For \(\ell=1,2\), let \(\psi_\ell:H_+\to H_-\) be an isometry of the induced boundary metrics, and assume that identifying \(H_+\) with \(H_-\) by \(\psi_\ell\) produces a smooth Riemannian manifold \((M_\ell,g_\ell)\).
Identify \(\partial M_\ell=B_+\dot\cup B_-\) through the fixed pieces.
Then
\[
\Lambda_{g_1,1}^{2R}=\Lambda_{g_2,1}^{2R}.
\]
\end{lemma}

\begin{proof}
Fix a boundary source \(f\), and let \(u_\ell\) be the corresponding solution on \(M_\ell\).
Write \(u_{\ell,\pm}=u_\ell|_{N_\pm}\), and let \(\nu_\pm\) be the outward unit normals of \(N_\pm\) along \(H_\pm\).
The smoothness of \(u_\ell\) across the seam is equivalent to the transmission conditions
\[
u_{\ell,+}
=
u_{\ell,-}\circ\psi_\ell,
\qquad
\partial_{\nu_+}u_{\ell,+}
=
-\bigl(\partial_{\nu_-}u_{\ell,-}\bigr)\circ\psi_\ell
\quad\text{on }H_+\times(0,2R).
\]
Since the initial data vanish and \(\operatorname{dist}_{g_\pm}(B_\pm,H_\pm)\geq R\), finite propagation \cite[Section~1.2]{feizmohammadi2021recovery} gives
\[
u_{\ell,\pm}=0
\quad\text{near }H_\pm
\quad\text{for }0<s<R.
\]
Thus the transmission traces for the two gluings agree for \(0<s<R\).
Set
\[
v_\pm=u_{1,\pm}-u_{2,\pm}.
\]
On each cut piece, \(v_\pm\) satisfies the homogeneous wave equation, has zero initial data and zero Dirichlet data on \(B_\pm\), and can acquire nonzero transmission data on \(H_\pm\) only for \(s\geq R\).
A second application of finite propagation gives
\[
v_\pm=0
\quad\text{near }B_\pm
\quad\text{for }0<s<2R.
\]
Consequently,
\[
\partial_\nu u_1=\partial_\nu u_2
\quad\text{on }(B_+\dot\cup B_-)\times(0,2R).
\]
Since \(f\) was arbitrary, the Dirichlet-to-Neumann maps agree.
The endpoint \(2R\), at which the first discrepancy may arrive, is not included in the observation interval.
\end{proof}

\begin{proof}[Proof of Theorem~\ref{thm:endpoint-gluing-obstruction}]
Let \(\Sigma\) be a closed connected \((n-1)\)-manifold carrying a metric \(k_0\) with a nontrivial isometry \(\iota\).
Choose smooth families \(k_\pm(r)\), \(0\leq r\leq R\), of Riemannian metrics on \(\Sigma\) such that \(k_\pm(r)=k_0\) near \(r=R\), and set
\[
N_\pm=[0,R]\times\Sigma,
\qquad
g_\pm=\mathrm{d}r^2+k_\pm(r).
\]
Form \(M_1\) by identifying \(\{R\}\times\Sigma\) in the two copies by \(\operatorname{Id}\), and form \(M_2\) by using \(\iota\).
The product form of the metrics near \(r=R\) ensures that both glued metrics are smooth.
Their boundaries are identified through
\[
B_+\dot\cup B_-
=
(\{0\}\times\Sigma)_+\dot\cup(\{0\}\times\Sigma)_-.
\]
Let \(\varphi:\partial M_1\to\partial M_2\) denote this fixed boundary identification.
Set
\[
H_\pm=(\{R\}\times\Sigma)_\pm.
\]

Figure~\ref{fig:endpoint-gluing-obstruction} illustrates the construction: the cut pieces \(N_+\) and \(N_-\), their metrics, and the observed boundary \(B_+\dot\cup B_-\) are fixed, while only the map identifying \(H_+\) with \(H_-\) is changed.
\begin{figure}[htbp]
\centering
\includegraphics[width=0.98\textwidth]{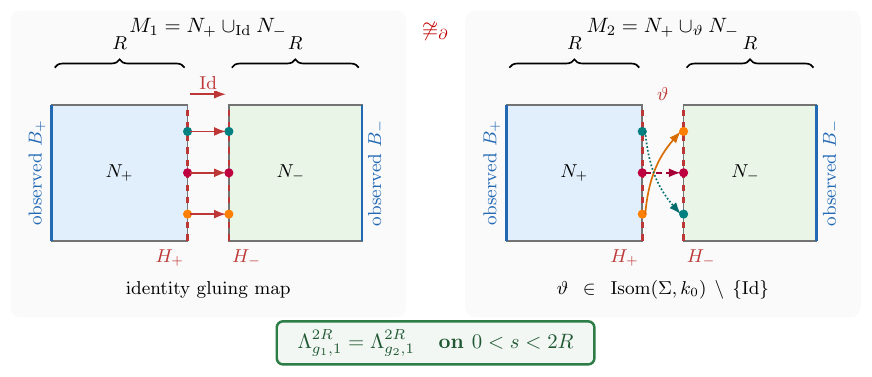}
\caption{Two identifications of \(H_+\) with \(H_-\).
The observed boundary and the metrics on \(N_+\dot\cup N_-\) are the same.
The identification map is either \(\operatorname{Id}\) or \(\iota\).
Since \(\operatorname{dist}_{g_\pm}(B_\pm,H_\pm)=R\), any discrepancy caused by the identification can first reach the observed boundary only at time \(2R\), the excluded endpoint of the measured interval.}
\label{fig:endpoint-gluing-obstruction}
\end{figure}

Let \(x=(r,y)\in N_\pm\).
Every curve from \(x\) to \(B_\pm\) has length at least \(r\), since \(|\nabla r|=1\), and the radial segment has length \(r\).
Any curve from \(x\) to \(B_\mp\) has length at least
\[
(R-r)+R=2R-r\geq r.
\]
It follows that
\[
\operatorname{dist}_{g_\ell}(x,\partial M_\ell)=r,
\]
and hence
\[
\operatorname{rad}(g_1)=\operatorname{rad}(g_2)=R.
\]

Suppose that an isometry \(F:M_1\to M_2\) extends the prescribed boundary identification.
Since distance to the boundary is preserved, \(F\) maps the hypersurface at distance \(R\) to itself.
It also preserves each component of its complement because it fixes \(B_+\) and \(B_-\) pointwise.
The restriction of \(F\) to either cut piece fixes \(B_\pm\) pointwise.
Its differential is the identity there: it fixes the tangent space of \(B_\pm\) and the inward unit normal.
An isometry of a connected Riemannian manifold is determined by its value and differential at one point, so \(F\) is the identity on each cut piece.
Compatibility with the two identification maps would then imply \(\iota=\operatorname{Id}\), a contradiction.

Applying Lemma~\ref{lem:endpoint-seam-propagation} with
\[
B_\pm=\{0\}\times\Sigma,
\qquad
H_\pm=\{R\}\times\Sigma,
\qquad
\psi_1=\operatorname{Id},
\qquad
\psi_2=\iota
\]
gives
\[
\Lambda_{g_1,1}^{2R}
=
P_\varphi^{-1}\Lambda_{g_2,1}^{2R}P_\varphi.
\]

The proof is complete.
\end{proof}

\begin{proof}[Proof of Corollary~\ref{cor:euclidean-recovery-consequences}]
Part~\textup{(i)} is Corollary~\ref{cor:fixed-euclidean-recovery}, and part~\textup{(ii)} is Corollary~\ref{cor:uniform-finite-time}.
\end{proof}

\begin{proof}[Proof of Theorem~\ref{thm:sharpness-obstructions}]
Parts~\textup{(i)}--\textup{(iv)} are, respectively, Theorem~\ref{thm:two-dimensional-conformal-obstruction}, Theorem~\ref{thm:causal-obstruction}, Corollary~\ref{cor:infinite-time-threshold}, and Theorem~\ref{thm:endpoint-gluing-obstruction}.
\end{proof}

\section*{Acknowledgments}

The work of L. Chen was partly supported by the National Natural Science Foundation of China (No.~12622108) and a grant from Beijing Institute of Technology (No.~2022CX01002).
The work of Y. Jiang was supported by the Hong Kong RGC Project JRFS2627-1S06.
The work of H. Liu was supported by the Hong Kong RGC General Research Funds (projects 11311122, 12301420, and 11300821).

\raggedbottom
\bibliographystyle{plain}
\bibliography{ref}

\end{document}